\documentclass[a4paper,10pt]{amsart}
\usepackage{amsmath,amsthm,amssymb,latexsym,enumerate,color,hyperref}
\usepackage{graphicx}
\usepackage{color}
\usepackage{eucal}
\numberwithin{equation}{section}

\begin{document}

\newtheorem{thm}{Theorem}[section]
\newtheorem{prop}[thm]{Proposition}
\newtheorem{lem}[thm]{Lemma}
\newtheorem{cor}[thm]{Corollary}

\newtheorem{rem}[thm]{Remark}

\newtheorem*{defn}{Definition}

\newcommand{\DD}{\mathbb{D}}
\newcommand{\NN}{\mathbb{N}}
\newcommand{\ZZ}{\mathbb{Z}}
\newcommand{\QQ}{\mathbb{Q}}
\newcommand{\RR}{\mathbb{R}}
\newcommand{\CC}{\mathbb{C}}
\renewcommand{\SS}{\mathbb{S}}

\renewcommand{\theequation}{\arabic{section}.\arabic{equation}}

\newcommand{\Erfc}{\mathop{\mathrm{Erfc}}}    
\newcommand{\supp}{\mathop{\mathrm{supp}}}    
\newcommand{\re}{\mathop{\mathrm{Re}}}   
\newcommand{\im}{\mathop{\mathrm{Im}}}   
\newcommand{\dist}{\mathop{\mathrm{dist}}}  
\newcommand{\link}{\mathop{\circ\kern-.35em -}}
\newcommand{\spn}{\mathop{\mathrm{span}}}   
\newcommand{\ind}{\mathop{\mathrm{ind}}}   
\newcommand{\rank}{\mathop{\mathrm{rank}}}   
\newcommand{\ol}{\overline}
\newcommand{\pa}{\partial}
\newcommand{\ul}{\underline}
\newcommand{\diam}{\mathrm{diam}}
\newcommand{\lan}{\langle}
\newcommand{\ran}{\rangle}
\newcommand{\tr}{\mathop{\mathrm{tr}}}
\newcommand{\diag}{\mathop{\mathrm{diag}}}
\newcommand{\dv}{\mathop{\mathrm{div}}}
\newcommand{\na}{\nabla}
\newcommand{\nr}{\Vert}

\newcommand{\al}{\alpha}
\newcommand{\be}{\beta}
\newcommand{\ga}{\gamma}  
\newcommand{\Ga}{\Gamma}
\newcommand{\de}{\delta}
\newcommand{\De}{\Delta}
\newcommand{\ve}{\varepsilon}
\newcommand{\fhi}{\varphi} 
\newcommand{\la}{\lambda}
\newcommand{\La}{\Lambda}    
\newcommand{\ka}{\kappa}
\newcommand{\vro}{\varrho}
\newcommand{\si}{\sigma}
\newcommand{\Si}{\Sigma}
\newcommand{\te}{\theta}
\newcommand{\zi}{\zeta}
\newcommand{\om}{\omega}
\newcommand{\Om}{\Omega}

\newcommand{\cA}{\mathcal{A}}
\newcommand{\cB}{\mathcal{B}}
\newcommand{\cC}{\mathcal{C}}
\newcommand{\cE}{\mathcal{E}}
\newcommand{\cG}{{\mathcal G}}
\newcommand{\cH}{{\mathcal H}}
\newcommand{\cI}{{\mathcal I}}
\newcommand{\cJ}{{\mathcal J}}
\newcommand{\cK}{{\mathcal K}}
\newcommand{\cL}{{\mathcal L}}
\newcommand{\cM}{\mathcal{M}}
\newcommand{\cN}{{\mathcal N}}
\newcommand{\cP}{\mathcal{P}}
\newcommand{\cR}{{\mathcal R}}
\newcommand{\cS}{{\mathcal S}}
\newcommand{\cT}{{\mathcal T}}
\newcommand{\cU}{{\mathcal U}}
\newcommand{\cX}{\mathcal{X}}

\title[Backus problem near the dipole] {Backus problem in geophysics: a
resolution \\ near the dipole 
in fractional Sobolev spaces
}

\author{Toru Kan} 
\address{Department of Mathematical Sciences, Osaka Prefecture University}
    \email{kan@ms.osakafu-u.ac.jp}
    \urladdr{}

\author{Rolando Magnanini} 
\address{Dipartimento di Matematica ed Informatica ``U.~Dini'',
Universit\` a di Firenze, viale Morgagni 67/A, 50134 Firenze, Italy.}
    \email{magnanini@unifi.it}
    \urladdr{http://web.math.unifi.it/users/magnanin}

\author{Michiaki Onodera} 
\address{Department of Mathematics, Tokyo Institute of Technology}
    \email{onodera@math.titech.ac.jp}
    \urladdr{}

\begin{abstract}
We consider Backus's problem in geophysics. This consists in reconstructing a harmonic potential outside the Earth when the intensity of the related field is measured on the Earth's surface. Thus, the boundary condition is (severely) nonlinear. The gravitational case is quite understood. It consists in the local resolution near a monopole, i.e. the potential generated by a point mass. In this paper, we consider the geomagnetic case. This consists in linearizing the field's intensity near the so-called dipole, a harmonic function which models the solenoidal potential of a magnet. The problem is quite difficult, because the resolving operator related to the linearized problem is generally unbounded. Indeed, existence results for Backus's problem  in this framework are not present in the literature.
\par
In this work, we locally solve the geomagnetic version of Backus's problem in the axially symmetric case. In mathematical terms, we show the existence of harmonic functions in the exterior of a sphere, with given (boundary) field's intensity sufficiently close to that of a dipole and which have the same axial symmetry of a dipole. We also show that unique solutions can be selected by prescribing the average of the potential on the equatorial circle of the sphere.  
\par
We obtain those solutions as series of spherical harmonics. The functional framework entails the use of fractional Sobolev Hilbert spaces on the sphere, endowed with a spectral norm. A crucial ingredient is the algebra structure of suitable subspaces. 
\end{abstract}

\date{\today}

\keywords{Backus problem in geophysics, geomagnetic potential, fully nonlinear boundary conditions, fractional Sobolev algebra}
    \subjclass[2010]{35J65, 86A25, 35C10, 35B07}

\maketitle

\raggedbottom

\section{Introduction}

One of the most interesting problems in geophysics is the determination of the gravitational or magnetic field 
of the Earth from surface measurements of its intensity. It is in fact much more convenient to measure field intensities rather than field directions. In mathematical terms, the problem can be formulated as follows.
\par
We shall represent the Earth's surface by the unit sphere $\cS\subset\RR^3$ (centered at the origin). We will denote by $\Om$ the exterior of $\cS$, i.e. the unbounded component of $\RR^3\setminus \cS$. Also, we suppose a positive continuous function $g$ is given on $\cS$. If $u$ represents the Earth's external scalar potential associated to the gravitational or magnetic field, then $|\na u|$ represents the field's intensity. 
Backus's problem then consists in finding solutions $u\in C^2(\Om)\cap C^1(\ol{\Om})$ of the following nonlinear boundary value problem: 
\begin{equation}
\label{backus-problem}
\De u =0 \ \mbox{ in } \ \Om, \quad |\na u|=g \ \mbox{ on } \ \cS, \quad u\to 0 \ \mbox{ as } \ |x|\to\infty.
\end{equation}
It is also important to specify sufficient conditions that ensure the uniqueness of a solution.
\par
A large part of the results known about this problem is contained in the pioneering work of G. E. Backus \cite{Ba1}. There, the problem of finding solutions of the first two equations in \eqref{backus-problem} is first examined for the two-dimensional case in a bounded domain and, by conformal mappings, also in exterior domains. A complete analysis is carried out: existence of solutions is proved, a severe lack of uniqueness is pointed out, and conditions which restore uniqueness are stated. Related to this case, questions about the regularity of solutions are investigated in \cite{Ma}, together with an analysis of the case in which the data $g$ may vanish at isolated points.
\par
In \cite{Ba1}, a condition for uniqueness of a solution for problem \eqref{backus-problem} is given in the gravitational case.  Non-uniqueness for the geomagnetic case is noted in \cite{Ba2}; \cite{JM} contains a general uniqueness condition for \eqref{backus-problem}, which includes both the gravitational and geomagnetic case. The solution of the linearized problem near the dipole is also constructed in \cite{JM} by means of an expansion in spherical harmonics. By a similar technique, the linearization near quadripoles is considered in \cite{AK} and solved. 
\par
For what concerns the existence of a solution of \eqref{backus-problem} in physical dimension, a first conclusive result is contained in \cite{Jo}. There, it is proved local existence and uniqueness of a solution in the gravitational case. It is obtained by linearizing  \eqref{backus-problem} at the so-called \textit{monopole},
$$
\Phi(x)=\frac1{|x|}, \ x\in\ol{\Om},
$$
which is nothing else than a normalized version of the fundamental solution of Laplace's equation.
 More precisely, one can prove existence and uniqueness of a solution of the form
\begin{equation}
\label{gravitational-ansatz}
u(x)=\Phi(x)+w(x), \ x\in\ol{\Om},
\end{equation}
provided $g$ is sufficiently close to $|\na\Phi|\equiv 1$ in the norm of $C^{0,\al}(\cS)$, $0<\al<1$, the space of \textit{$\al$-H\"older continuous} functions on $\cS$. In fact,  plugging the \textit{ansatz} \eqref{gravitational-ansatz} into
\eqref{backus-problem} gives that $w$ must solve the problem:
\begin{equation}
\label{near-monopole}
\De w=0 \ \mbox{ in } \ \Om, \ w_\nu+\frac12\,|\na w|^2=\frac{g^2-1}{2} \ \mbox{ on } \ \cS, \ w\to 0 \ \mbox{ as } \ |x|\to\infty.
\end{equation} 
Here, $\nu$ denotes the unit normal to $\cS$, exterior to $\Om$. 
\par
In \cite{Jo} (see also \cite{SS}), the nonlinear problem  \eqref{near-monopole} is solved by a Neumann series, essentially based on the following fixed point argument. 
In fact, one can introduce a (nonlinear) operator $\cT_\Phi$ by formally setting
$$
\cT_\Phi[f]=|\na v|^2,
$$
where $v$ is the solution of the Neumann problem:
\begin{equation}
\label{neumann}
\De v=0 \ \mbox{ in } \ \Om, \quad v_\nu=\frac12\,f \ \mbox{ on } \ \cS, \quad v\to 0 \ \mbox{ as } \ |x|\to\infty.
\end{equation} 
Since we know that \eqref{neumann} always has a unique solution $v\in C^{1,\al}(\ol{\Om})$ for any $f\in C^{0,\al}(\cS)$, then $\cT_\Phi$ turns out to be well-defined as an operator on $C^{0,\al}(\cS)$ into itself.
The problem \eqref{near-monopole} is thus converted into the fixed-point equation:
$$
f+\cT_\Phi[f]=g^2-1 \ \mbox{ on } \ \cS.
$$
This can be uniquely solved by a function $f^*\in C^{0,\al}(\cS)$, provided $g^2-1$ is sufficiently small in the norm of $C^{0,\al}(\cS)$. A solution of \eqref{backus-problem} is therefore obtained by means of \eqref{gravitational-ansatz}, where $w$ is the solution of \eqref{neumann} corresponding to $f^*$. 
\par
We conclude our review of known results with a couple of papers, \cite{DDO1, DDO2},  which provide a genuinely nonlinear approach to problem \eqref{backus-problem}. In \cite{DDO1}, \eqref{backus-problem} is converted into a boundary value problem in the unit ball $B$:
\begin{equation}
\label{problem-for-U}
\De U=0 \ \mbox{ in } \ B, \quad (U+U_\nu)^2+|\na_\cS U|^2=g^2 \ \mbox{ on } \ \cS,
\end{equation}
where $\na_\cS U$ denotes the tangential gradient of $U$ on $\cS$. Here, $U$ is the Kelvin's transformation of $u$, which is such that
$$
u(x)=|x|^{-1} U(x/|x|^2) \ \mbox{ for } \ x\in\Om.
$$
Then, a solution of \eqref{problem-for-U} is obtained by solving the following boundary value problem:
 \begin{equation}
\label{backus-problem-ball}
\De U=0 \ \mbox{ in } \ B, \quad U+U_\nu=\sqrt{(g^2-|\na_\cS U|^2)_+} \ \mbox{ on } \ \cS.
\end{equation}
Since $U+U_\nu=-u_\nu$ on $\cS$, the corresponding solution $u$ of \eqref{backus-problem} is such that $u_\nu\le 0$ on $\cS$. It turns out that harmonic functions in $\Om$, which vanish at infinity and are subject to the constraint $u_\nu\le 0$ on $\cS$, satisfy some sort of comparison principle. This property is then instrumental to the definition of maximal and minimal solutions of \eqref{backus-problem} such that $u_\nu\le 0$ on $\cS$. This fact allows the construction of a suitably defined viscosity solution of \eqref{backus-problem-ball}.
In \cite{DDO2}, a numerical scheme to construct a maximal solution is proposed.

\medskip

The aim of this paper is to study the local resolution of the geomagnetic case, i.e. the (local) existence and uniqueness of solutions of \eqref{backus-problem} near the \textit{dipole} defined by
$$
d(x)=\frac{x_3}{|x|^3} \ \mbox{ for } \ x\in\ol{\Om}.
$$
In spherical coordinates $x=r\,(\cos\te \cos\fhi, \cos\te \sin\fhi, \sin\te)$, $d$ can be written as
$$
d=\frac{\sin\te}{r^2}.
$$
Here $r=|x|$, $-\pi/2\le\te\le\pi/2$ is the \textit{latitude}, and $-\pi\le\fhi<\pi$ is the \textit{longitude} on the Earth $\cS$. To the best of our knowledge, existence results for this problem are not present in the literature.
\par
Thus, similarly to the gravitational case, we linearize problem \eqref{backus-problem} by setting $u=d+w$ and obtain the following problem for $w$:
\begin{eqnarray}
\label{near-dipole}
&&\De w=0 \ \mbox{ in } \ \Om, \nonumber \\ 
&&\na d\cdot\na w+\frac12\,|\na w|^2=\frac{g^2-|\na d|^2}{2} \ \mbox{ on } \ \cS, \\ 
&&w\to 0 \ \mbox{ as } \ |x|\to\infty. \nonumber
\end{eqnarray}
\par
Note that, being as $\nu(x)=-x$, we have that
\begin{equation}
\label{gradient-d}
\na d(x)=\tau(x)+2 x_3\,\nu(x) \ \mbox{ for } \ x\in \cS.
\end{equation}
The vector filed
\begin{equation*}
\tau(x)=e_3+x_3\,\nu(x) \ \mbox{ for } \ x\in \cS,
\end{equation*}
is the tangential to $S$ and is obtained by projecting $e_3=(0,0, 1)$ on the tangent plane of $\cS$ at $x\in \cS$. Notice that $\na d(x)$ has intensity $|\na d(x)|=\sqrt{1+3 x_3^2}$ for $x\in\cS$, and points outward to the Earth's surface on the south hemisphere, becomes tangential on the equator $\cE=\{ x\in \cS: x_3=0\}$, and points inward on the north hemisphere. This behavior of $\na d$ tells us that neither $d$ nor any solution of \eqref{backus-problem} sufficiently close to $d$ falls within the class of solutions studied in \cite{DDO1, DDO2}. 
\par
Proceeding as in the monopole case gives the \textit{(irregular) oblique derivative} problem:
\begin{equation}
\label{oblique}
\De v=0 \ \mbox{ in } \ \Om, \quad \na d\cdot\na v=\frac12\,f \ \mbox{ on } \ \cS, \quad v\to 0 \ \mbox{ as } \ |x|\to\infty.
\end{equation} 
Differently from \textit{regular} oblique derivative problems, in which the relevant directional derivative  is controlled by a vector field that points either inward or outward \textit{on the whole boundary}, the irregular ones may present at least two setbacks. These are caused precisely by a change of direction, as described for $\na d$. 
\par
The former is a non-trivial lack of uniqueness. In fact, uniqueness can be obtained only by prescribing Dirichlet boundary values on the subset of the boundary in which the relevant vector field becomes tangential (the equator $\cE$ in the problem at stake). In other words, the kernel of the resolvent operator has infinite dimension. 
\par
The latter is the so-called \textit{loss of derivatives}. In fact, it may happen that suitably determined unique solutions of oblique boundary value problems with, say, $C^{0,\al}$-regular boundary data, do not gain $C^{1,\al}$-regularity up to the boundary, as it does happen for the Neumann problem or the regular oblique derivative problem (see for instance \cite{Al,Wi1,Wi2}). In other words, the linear operator on $C^{0,\al}(\cS)$ associating the oblique derivative data to the (trace on the boundary of the) solution of the problem may be unbounded. A similar behavior also occurs if we try to solve the oblique derivative problem in the scale of fractional Sobolev spaces $H^s(\cS)$ (this fact can be detected by an inspection of the solution obtained in \cite{JM}).
\par
Besides causing a loss of surjectivity of the relevant operator, more importantly, the loss of derivatives disrupts the iterative scheme on which a classical contraction argument is based.
Thus, the strategy of defining an operator $\cT_{d,h}$ by first setting
$$
\cT_{d,h}[f]=|\na v|^2,
$$
where, for some fixed $h:\cE\to\RR$, $v$ is the solution of \eqref{oblique}, subject to the Dirichlet-type condition
\begin{equation}
\label{dirichlet-equator}
v=h \ \mbox{ on } \ \cE,
\end{equation}
and then solving the equation
\begin{equation*}
f+\cT_{d,h}[f]=g^2-|\na d|^2,
\end{equation*}
may miserably fail.
\par
To by-pass these difficulties, a standard idea would be to use the \textit{Nash-Moser implicit function theorem}. This often works when a loss of derivatives occurs. 
The second author and M. C. Jorge have tried this pathway, but with no success. The main difficulty is the  lack of sufficiently precise estimates for the relevant oblique derivative problems involved. In fact, in such approach, one would need to precisely control estimates not only for the solution of \eqref{oblique}, but also for a class of oblique derivative problems obtained by perturbing $\na d$.
\par
In this paper, we turn back to a fixed-point approach and obtain local existence near the dipole for the nonlinear problem \eqref{backus-problem} in the case the boundary data $g$ is \textit{axially symmetric} around the Earth's axis. This result is obtained in the framework of fractional Sobolev spaces and is made possible from the discovery that the relevant oblique derivative problem \eqref{oblique} no longer loses derivatives in presence of axially symmetric data.
Hence, a fixed-point scheme still  works for problem \eqref{near-dipole}.
\par
From a technical viewpoint, we construct solutions of  \eqref{oblique}-\eqref{dirichlet-equator}  by means of series expansions of spherical harmonics as done in \cite{JM}.
This time, we trace more carefully the dependence on the data $f$ and $h$ of the coefficients of the relevant expansion. The aim is to obtain precise estimates for the operator $\cT_{d,h}$ in the scale of Sobolev spaces $H^s(\cS)$. It turns out that $\cT_{d,h}$ is well-defined as an operator on the subspace $H^s_{\rm ax}(\cS)$ of axially symmetric functions on $\cS$. (For a precise definition of $H^s(\cS)$ and $H^s_{\rm ax}(\cS)$, see Section \ref{sec:spherical-harmonics}.)
\par
Our main result is then the following existence and uniqueness theorem.

\begin{thm}
\label{thm:existence-Sobolev}
Suppose that $g \in H^s_{\rm ax}(\cS)$ for some $s>1$ and that $h\equiv h_0$ on $\cS$ with $h_0\in\RR$.
If $\|g -|\na d|\|_{H^s(\cS)}$ and $|h_0|$ are sufficiently small,
then problem \eqref{backus-problem} has a unique solution 
$u \in C^2(\Om) \cap C^1(\ol{\Om})$
satisfying \eqref{dirichlet-equator}. 
Moreover, we have that 
$u|_{\cS} \in H^{s+1}_{\rm ax}(\cS)$ and
\begin{equation}\label{difference-estimate}
\nr u-d\nr_{H^{s+1}(\cS)}\le C\left(\nr g-|\nabla d|\nr_{H^s(\cS)}+|h_0|\right) 
\end{equation}
for some constant $C>0$.
\end{thm}
From the continuous embedding of $H^s(\cS)$ into $C(\cS)$ 
(see Proposition~\ref{embedding-Hs-C} below) 
and the maximum principle for harmonic functions,
we see that \eqref{difference-estimate} holds 
if the left-hand side is replaced with $\nr u-d\nr_{C(\ol{\Om})}$.

We conclude this introduction with one more important technical remark about the proof of Theorem \ref{thm:existence-Sobolev}. 
\par
In fact, it should be noticed that
the Hilbert-space structure we adopt for $H^s(\cS)$ and $H^s_{\rm ax}(\cS)$ is based on 
an inner product of a spectral-type. In other words, the relevant inner product is defined in terms of the coefficients in the spherical-harmonics series expansions  of the functions at stake. In order to deal with problem \eqref{near-dipole}, which shows a \textit{quadratic nonlinearity} in the gradient, a \textit{Banach-algebra structure} for $H^s(\cS)$ is desirable. Such a structure is rather easily derived in case of inner products based on pointwise multiplication of functions. However, it is not the case for the inner product we choose in this paper. As a matter of fact, the proof of a Banach-algebra structure for $H^s(\cS)$ seems not available in the literature. Thus, in Theorem \ref{prop-products}, we provide our own proof for the case of axially symmetric functions. This is based on the series expansions in spherical harmonics of products of spherical harmonics, and the so-called \textit{Wigner $3j$-symbols}.
(In Proposition \ref{prop:weighted spaces}, we show that such a structure also holds in other instances.) 
\par
The paper is organized as follows. We begin with Section \ref{sec:spherical-harmonics}, in which we collect all the functional analytical results on the spaces $H^s(\cS)$ and $H^s_{\rm ax}(\cS)$ useful for our purposes. Then, in Section \ref{sec:series-solution}, we derive the appropriate estimates for the relevant oblique derivative problem. The proof of Theorem \ref{thm:existence-Sobolev} is given in Section \ref{sec:backus-problem}.

\smallskip


\section{The fractional spectral Sobolev space $H^s(\cS)$ \\
and its subspace $H^s_{\rm ax}(\cS)$}
\label{sec:spherical-harmonics}

In this section, we collect old and new results on the spectral Hilbert space $H^s(\cS)$.

\subsection{Spherical harmonics and the space $H^s(\cS)$}
As mentioned in the introduction, we adopt a system of spherical coordinates by setting
$$
x=r\,(\cos\te \cos\fhi, \cos\te \sin\fhi, \sin\te),  \ r>0, \ -\frac{\pi}{2}\le \te \le \frac{\pi}{2}, \ -\pi\le\fhi<\pi;
$$
we call $\te$ the \textit{latitude} and $\fhi$ the \textit{longitude} on the earth surface $\cS$.
With this parametrization, the surface element on $\cS$ is given by
$dS_x=\cos\te\,d\te\,d\fhi$.
\par
We denote by $\{ Y_l^m(\te,\fhi)\}_{|m|\le l, l=0, 1,\dots}$ the spherical harmonics system of functions. We have that
$$
Y_l^m(\te,\fhi)=
\al_l^m P_l^{|m|} (\sin \te)\, e^{im \fhi}
\qquad
(|m| \le l, \ l=0, 1, \dots),
$$
where $P_l^m(z)$ is the associated Legendre polynomial of degree $l$ and order $m$,
and $\al_l^m$ is defined by
$$
\al_l^m=(-1)^{\frac{m+|m|}{2}} \sqrt{\frac{(2l+1)(l-|m|)!}{4\pi(l+|m|)!}}.
$$
It is well-known that $\{ Y_l^m(\te,\fhi)\}_{|m|\le l, l=0, 1,\dots}$ forms an orthonormal basis of $L^2(\cS)$,
that is, the equality 
\begin{eqnarray*}
&&\psi(\te,\fhi)=\sum_{l=0}^\infty \sum_{m=-l}^l \widehat{\psi}_l^m Y_l^m(\te,\fhi)
\ \mbox{ with } \\
&&\widehat{\psi}_l^m=\int_{-\pi}^\pi \int_{-\frac{\pi}{2}}^{\frac{\pi}{2}} 
\psi(\te,\fhi) \, \overline{Y_l^m(\te,\fhi)} \cos \te\, d\te d\fhi 
\end{eqnarray*}
holds in $L^2(\cS)$ for any function $\psi\in L^2(\cS)$.
\par
For any non-negative real number $s$, 
we define the \textit{fractional Sobolev space}:
\begin{equation*}
H^s(\cS)=\biggl\{ \psi\in L^2(\cS):
\sum_{l=0}^\infty \sum_{m=-l}^l (l+1)^{2s} |\widehat{\psi}_l^m|^2 <+\infty \biggr\}.
\end{equation*}
Thus, we have that $H^0(\cS)=L^2(\cS)$ and we know that $H^s(\cS)$ is a Hilbert space endowed with the inner product
\begin{equation*}
\lan \psi,\phi \ran_{H^s(\cS)}
=\sum_{l=0}^\infty \sum_{m=-l}^l (l+1)^{2s} \widehat{\psi}_l^m \overline{\widehat{\phi}_l^m}.
\end{equation*}
Similarly, we set
\begin{align*}
&H^s(\cE)=\biggl\{ \Psi\in L^2(\cE):
\sum_{m=-\infty}^\infty (|m|+1)^{2s} |\widehat{\Psi}_m|^2 <+\infty \biggr\},
\\
&\mbox{where } \ \widehat{\Psi}_m=\frac{1}{2\pi} \int_{-\pi}^\pi \Psi (\fhi) \, e^{-im\fhi} \, d\fhi,
\end{align*}
and write the norm on $H^s(\cE)$ as
\begin{equation*}
\| \Psi \|_{H^s(\cE)}
=\sqrt{\sum_{m=-\infty}^\infty (|m|+1)^{2s} |\widehat{\Psi}_m|^2}.
\end{equation*}
Then, any function $\Psi \in H^s(\cE)$ can be expanded in $H^s(\cE)$ as
$$
\Psi (\fhi)=\sum_{m=-\infty}^\infty\widehat{\Psi}_m e^{im\fhi}.
$$
\par
The following properties of $H^s(\cS)$ will be useful in the sequel.
\begin{prop}\label{embedding-Hs-C}
Set $s>1$. Then, the Sobolev space $H^s(\cS)$ is continuously embedded into $C(\cS)$. 
\end{prop}

\begin{proof}
It is known that the identity
\begin{multline}
\label{addition-theorem}
\sum_{m=-l}^l Y_l^m(\theta_1,\varphi_1) \overline{Y_l^m(\theta_2,\varphi_2)}
\\
=\frac{2l+1}{4\pi}P_l(\sin \theta_1 \sin \theta_2 +\cos \theta_1 \cos \theta_2 \cos (\varphi_2-\varphi_1))
\end{multline}
holds for any $l=0, 1, \dots$, $-\pi/2\le \te_1, \te_2 \le \pi/2$, and $-\pi\le\fhi_1, \fhi_2<\pi$ (see \cite[(16.57), (16.59)]{AWH}). 
In particular, this identity and the fact that $P_l(1)=1$ give
\begin{equation}
 \sum_{m=-l}^l|Y_l^m(\theta,\varphi)|^2=\frac{2l+1}{4\pi}.
\label{Ylm-identity}
\end{equation}
Hence, by the Cauchy-Schwarz inequality, we have that
\begin{equation*}
 \sum_{m=-l}^l\left|\widehat{u}_l^mY_l^m(\theta,\varphi)\right|
 \leq \sqrt{\frac{2l+1}{4\pi}} \sqrt{\sum_{m=-l}^l\left|\widehat{u}_l^m\right|^2}.
\end{equation*}
Taking the sum in $l$ and using the Cauchy-Schwarz inequality again,
we obtain
\begin{align*}
 \sum_{l=0}^\infty \sum_{m=-l}^l\left|\widehat{u}_l^mY_l^m(\theta,\varphi)\right|
 &\leq \sqrt{\sum_{l=0}^\infty \frac{(l+1)^{-2s}(2l+1)}{4\pi}}
 \sqrt{\sum_{l=0}^\infty \sum_{m=-l}^l (l+1)^{2s}\left|\widehat{u}_l^m\right|^2}\\
 &\leq \sqrt{\sum_{l=0}^\infty \frac{(l+1)^{1-2s}}{2\pi}} \|u\|_{H^s(\cS)}.
\end{align*}
Thus, the series (of continuous functions) on the left-hand side defines a continuous function, since it converges uniformly and absolutely for $s>1$. 
\end{proof}

\begin{prop}
\label{prop:C-infty-in-Hs}
Let $k$ be any natural number.  
Then, $C^k(\cS)$ is continuously embedded into $H^k(\cS)$. 
\end{prop}

\begin{proof}
Assume $u \in C^k(\cS)$.
Let $\nabla_{\cS}$ and $\Delta_{\cS}$ 
denote the gradient and the Laplace-Beltrami operator on $\cS$, respectively. 
In order to prove the proposition,
we verify that the equalities 
\begin{gather}
\label{u-Hk-identity1}
\sum_{l=0}^\infty\sum_{m=-l}^l \left[l(l+1)\right]^{2j} |\widehat{u}_l^m|^2 
=\left\| (-\Delta_{\cS})^j u\right\|_{L^2(\cS)}^2,
\\
\label{u-Hk-identity2}
\sum_{l=0}^\infty\sum_{m=-l}^l \left[l(l+1)\right]^{2j+1} |\widehat{u}_l^m|^2 
=\left\| \nabla_{\cS} \left[ (-\Delta_{\cS})^ju\right] \right\|_{L^2(\cS)}^2
\end{gather}
hold for any nonnegative integer $j$ 
with $2j \le k$, $2j+1 \le k$, respectively.

By repeated integration by parts and the fact that
\begin{equation*}
 -\Delta_{\cS} Y_l^m=l(l+1)\,Y_l^m,
\end{equation*}
the Fourier-Laplace coefficient of $(-\De_{\cS})^j u$ is computed as
\begin{equation}
\label{FL-Deltak2u}
\begin{aligned}
 \left[\widehat{(-\Delta_{\cS})^j u}\right]_l^m&=\int_{\cS}[(-\Delta_{\cS})^ju]\, \overline{Y_l^m}\,dS\\
 &=\int_{\cS}u\,[(-\Delta_{\cS})^j\overline{Y_l^m}]\,dS
 =\left[l(l+1)\right]^j\widehat{u}_l^m.
\end{aligned}
\end{equation}
Hence \eqref{u-Hk-identity1} follows.
To derive \eqref{u-Hk-identity2},
we set $w=(-\Delta_{\cS})^ju$.
For the moment, we suppose that $w \in C^2(\cS)$.
Then, we have that
$$
\left\| \nabla_{\cS} w\right\|_{L^2(\cS)}^2
=-\int_{\cS} w \Delta_{\cS} w dS
=-\sum_{l=0}^\infty\sum_{m=-l}^l \overline{\widehat{w}_l^m}
\int_{\cS} \overline{Y_l^m} \Delta_{\cS} w dS,
$$
where we used integration by parts and the expansion
$$
w=\sum_{l=0}^\infty\sum_{m=-l}^l \widehat{w}_l^m Y_l^m
=\sum_{l=0}^\infty\sum_{m=-l}^l \overline{\widehat{w}_l^m} \overline{Y_l^m}.
$$
From the definition of $w$ and \eqref{FL-Deltak2u},
we see that
$$
-\overline{\widehat{w}_l^m} \int_{\cS} \overline{Y_l^m} \Delta_{\cS} w dS
=\overline{\left[\widehat{(-\Delta_{\cS})^j u}\right]_l^m}
\left[\widehat{(-\Delta_{\cS})^{j+1} u}\right]_l^m
=\left[l(l+1)\right]^{2j+1} |\widehat{u}_l^m|^2,
$$
and therefore
$$
\left\| \nabla_{\cS} w\right\|_{L^2(\cS)}^2
=\sum_{l=0}^\infty\sum_{m=-l}^l \left[l(l+1)\right]^{2j+1} |\widehat{u}_l^m|^2.
$$
By approximation, this equality also holds if $w \in C^1(\cS)$.
We thus obtain \eqref{u-Hk-identity2}.

Using \eqref{u-Hk-identity1} and \eqref{u-Hk-identity2}, 
we deduce that
\begin{align*}
 \|u\|_{H^{k}(\cS)}^2&\leq C\sum_{l=0}^\infty\sum_{m=-l}^l\left( \left[l(l+1)\right]^k+1\right)|\widehat{u}_l^m|^2
\\
&=\left\{
\begin{aligned}
&C\left( \left\|(-\Delta_{\cS})^{\frac{k}{2}}u\right\|_{L^2(\cS)}^2+\left\|u\right\|_{L^2(\cS)}^2\right)
\quad \mbox{if } \ k \ \mbox{ is even},
\\
&C\left( \left\| \nabla_{\cS} \left[ (-\Delta_{\cS})^{\frac{k-1}{2}} u\right]\right\|_{L^2(\cS)}^2+\left\|u\right\|_{L^2(\cS)}^2\right)\quad \mbox{if } \ k \ \mbox{ is odd}
\end{aligned}
\right.
\\
&\le C\| u\|_{C^k(\cS)}^2,
\end{align*}
where $C>0$ is some constant.
This proves the proposition.
\end{proof}

\smallskip

\subsection{The subspace $H^s_{\rm ax}(\cS)$ and its Banach-algebra structure}

In our analysis, the subspace of \textit{axially symmetric} functions defined by
$$
H^s_{\rm ax}(\cS)=\bigl\{\psi\in H^s(\cS): \psi \mbox{ does not depend on $\fhi$}\bigr\}
$$
will play a crucial role. It is clear that
$$
H^s_{\rm ax}(\cS)=\bigl\{\psi\in H^s(\cS): \widehat{\psi}^m_l=0, \,  1\le |m|\le l, \, l=1, 2, \dots  \bigr\}.
$$

\smallskip

This section is the technical core of this paper.
In fact, observe that, for the solvability of the nonlinear problem \eqref{near-dipole}, we need to deal with the quadratic term $|\na w|^2$. In other words, in the functional framework adopted, we must be sure that the product of two functions in the relevant space still belongs to the same space. We shall show that the subspace $H^s_{\rm ax}(\cS)$ enjoys this property, i.e. it is a Banach algebra with respect to the pointwise product. 
\par
To this aim, we recollect some notations and results about products of spherical harmonics.
We recall (see \cite[Appendix III]{Th}) that the product of two spherical harmonics $Y_{l_1}^{m_1}, Y_{l_2}^{m_2}$ is represented by the formula:
\begin{multline*}
Y_{l_1}^{m_1}(\te,\fhi)Y_{l_2}^{m_2}(\te,\fhi)\\
=\sum_{l=0}^\infty\!\sum_{m=-l}^{l}(-1)^m\sqrt{\frac{(2l_1\!+\!1)(2l_2\!+\!1)(2l\!+\!1)}{4\pi}}
\!\begin{pmatrix}
  l_1 &\!\!l_2 &\! \!l\!\\
  m_1 &\!\! m_2 &\!\!\! -m
 \end{pmatrix} 
 \!\begin{pmatrix}
  l_1 &\! l_2 &\! l\\
  0 &\! 0 &\! 0
 \end{pmatrix}
 Y_{l}^{m}(\te,\fhi).
\end{multline*}
In this formula, the so-called \textit{Wigner $3$-$j$ symbol} is defined by
\begin{multline*}
 \begin{pmatrix}
  l_1 &\!\! l_2 &\!\! l_3\\
  m_1 &\!\! m_2 &\!\! m_3
 \end{pmatrix}
 =\sqrt{\frac{(-l_1\!+\!l_2\!+\!l_3)!(l_1\!-\!l_2\!+\!l_3)!(l_1\!+\!l_2\!-\!l_3)!(l_3\!-\!m_3)!(l_3\!+\!m_3)!}{(l_1\!+\!l_2\!+\!l_3\!+\!1)!(l_1\!-\!m_1)!(l_1\!+\!m_1)!(l_2\!-\!m_2)!(l_2\!+\!m_2)!}} \\
 \times\sum_k\frac{(-1)^{k+l_1+m_2-m_3}(l_2\!+\!l_3\!+\!m_1\!-\!k)!(l_1\!-\!m_1\!+\!k)!}{k!(-l_1\!+\!l_2\!+\!l_3\!-\!k)!(l_3\!-\!m_3\!-\!k)!(l_1\!-\!l_2\!+\!m_3\!+\!k)!},
\end{multline*}
if $m_1+m_2+m_3=0$, $|l_1-l_2|\leq l_3\leq l_1+l_2$, $|m_1|\leq l_1$, $|m_2|\leq l_2$ and $|m_3|\leq l_3$;  the symbol is set to be zero otherwise. 
The summation in the formula is taken over all integers $k$ for which all the factorials in the sum have nonnegative arguments. 
\par
By using the product formula, we find that the product $uv$ of $u,v\in H^s(\cS)$ has the following Fourier-Laplace coefficients:
\begin{equation*}
 \widehat{uv}_l^m=\sum_{l_1=0}^\infty\sum_{m_1=-l_1}^{l_1}\sum_{l_2=0}^\infty\sum_{m_2=-l_2}^{l_2}\sqrt{(2l_1\!+\!1)(2l_2\!+\!1)(2l\!+\!1)}\,W^{l_1,l_2,l}_{m_1,m_2,m}\widehat{u}_{l_1}^{m_1}\widehat{v}_{l_2}^{m_2}, 
\end{equation*}
where 
\begin{equation}
\label{w_ijl}
 W^{l_1,l_2,l}_{m_1,m_2,m}:=\frac{(-1)^m}{\sqrt{4\pi}}
 \begin{pmatrix}
  l_1 &\!\! l_2 &\!\! l\\
  m_1 &\!\! m_2 &\!\! -m
 \end{pmatrix}
 \!\begin{pmatrix}
  l_1 &\! l_2 &\! l\\
  0 & 0 & 0
 \end{pmatrix}. 
\end{equation}

Our proof of the Banach-algebra property of $H^s_{\rm ax}(\cS)$ is based on an $l$-sum relation (see \cite[(7.61)]{Th}) satisfied by the Wigner $3$-$j$ symbols, that is
\begin{equation*}
 \sum_{l_j=0}^\infty(2l_j+1)
 \begin{pmatrix}
  l_1 &\!\! l_2 &\!\! l_3\\
  m_1 &\!\! m_2 &\!\! m_3
 \end{pmatrix}
 ^2=1 \quad (j=1,2,3),
\end{equation*}
if $m_1+m_2+m_3=0$ and $|m_i|\leq l_i$ for $i\neq j$. 
This yields that
\begin{equation}
\label{l-sum}
\begin{aligned}
 \sup_{l_2,m_2,l,m}\sum_{l_1=0}^\infty\sum_{m_1=-l_1}^{l_1}(2l_1+1)|W^{l_1,l_2,l}_{m_1,m_2,m}|&\leq\frac{1}{\sqrt{4\pi}},\\
 \sup_{l_1,m_1,l,m}\sum_{l_2=0}^\infty\sum_{m_2=-l_2}^{l_2}(2l_2+1)|W^{l_1,l_2,l}_{m_1,m_2,m}|&\leq\frac{1}{\sqrt{4\pi}},\\
 \sup_{l_1,m_1,l_2,m_2}\sum_{l=0}^\infty\sum_{m=-l}^{l}(2l+1)|W^{l_1,l_2,l}_{m_1,m_2,m}|&\leq\frac{1}{\sqrt{4\pi}},
\end{aligned}
\end{equation}
by the Cauchy-Schwarz inequality. 

\smallskip

We conclude our preliminaries with the following simple result.

\begin{lem}
\label{lem-weighted_l1}
For any $\sigma<s-1/2$, there is a constant $C>0$ such that
\begin{equation*}
 \sum_{l=0}^\infty(2l+1)^\sigma |\widehat{u}_l^0|\leq C\|u\|_{H^s(\cS)}
\end{equation*}
for any $u\in H^s_{\rm ax}(\cS)$. 
\end{lem}

\begin{proof}
The Cauchy-Schwarz inequality implies that
\begin{equation*}
 \sum_{l=0}^\infty(2l+1)^\sigma|\widehat{u}_l^0|\leq\sqrt{\sum_{l=0}^\infty(2l+1)^{2(\sigma-s)}}\sqrt{\sum_{l=0}^\infty(2l+1)^{2s}|\widehat{u}_l^0|^2}, 
\end{equation*}
where the first series on the right hand side converges if $2(\sigma-s)<-1$. 
\end{proof}

We are now ready to state and prove the main result of this section.

\begin{thm}[$H^s_{\rm ax}(\cS)$ is a Banach algebra]
\label{prop-products}
Let $s>1$. 
If $u,v\in H^s_{\rm ax}(\cS)$, then $uv\in H^s_{\rm ax}(\cS)$ and 
\begin{equation*}
 \|uv\|_{H^s(\cS)}\leq C\|u\|_{H^s(\cS)}\|v\|_{H^s(\cS)}
\end{equation*}
for some constant $C>0$ independent of $u,v$. 
\end{thm}

\begin{proof}
Let us simply write $\widehat{u}_l, \widehat{v}_l, \widehat{uv}_l$ for $\widehat{u}_l^0, \widehat{v}_l^0, \widehat{uv}_l^0$, respectively. 
Since $W^{l_1,l_2,l}_{0,0,0}$ is nonzero only when $l\leq l_1+l_2$, and in this situation $l_1\leq l_2$ implies $l\leq 2l_2$, while $l_1\geq l_2$ implies $l\leq 2l_1$, we have
\begin{align*}
 (l+1)^s|\widehat{uv}_l|&\leq\sqrt{2}\sum_{l_1=0}^\infty\sum_{l_2=l_1}^\infty\sqrt{2l_1+1}(2l_2+1)^{s}(l+1)|W^{l_1,l_2,l}_{0,0,0}||\widehat{u}_{l_1}||\widehat{v}_{l_2}|\\
 &\quad +\sqrt{2}\sum_{l_2=0}^\infty\sum_{l_1=l_2}^\infty(2l_1+1)^{s}\sqrt{2l_2+1}(l+1)|W^{l_1,l_2,l}_{0,0,0}||\widehat{u}_{l_1}||\widehat{v}_{l_2}|.
\end{align*}
In what follows, we denote the two summands in this formula by $I_l$ and $J_l$, respectively. \par
Now, Lemma \ref{lem-weighted_l1} shows that
\begin{equation}
\label{ineq-weighted_l1}
 \sum_{l_1=0}^\infty\sqrt{2l_1+1}|\,\widehat{u}_{l_1}|\leq C\|u\|_{H^s(\cS)}, \quad \sum_{l_2=0}^\infty\sqrt{2l_2+1}|\,\widehat{v}_{l_2}|\leq C\|v\|_{H^s(\cS)}. 
\end{equation}
Thus, we use Lemma \ref{lem-convolution_inequality} with the settings
\begin{eqnarray*}
&&p_k=\sqrt{2k+1}\,|\widehat{u}_k|, \quad q_k=(2k+1)^{s}|\widehat{v}_k|, \ \mbox{ and } \\ 
&&r_{i, j,k}=
\begin{cases}
(k+1)|W_{0,0,0}^{i, j ,k}| \ \mbox{ for } \ i\le j, \\
0 \ \mbox{ for } \ i>j,
\end{cases}
\end{eqnarray*}
and, by recalling the $l$-sum relation \eqref{l-sum}, we see that
\begin{align*}
 \sqrt{\sum_{l=0}^\infty |I_l|^2}\leq\frac{C}{\sqrt{2\pi}}\|u\|_{H^s(\cS)}\|v\|_{H^s(\cS)}.
\end{align*}
A similar formula can be obtained for the terms $J_l$. Therefore, the proof is completed. 
\end{proof}

For future reference, we conclude this section by showing that the
algebra structure still holds for the entire space $H^s(\cS)$ if
$s>3/2$.

\begin{prop}
\label{prop:weighted spaces}
If $u,v\in H^s(\cS)$ with $s>3/2$, then $uv\in H^s(\cS)$ and
\begin{equation*}
 \|uv\|_{H^s(\cS)}\leq C\|u\|_{H^s(\cS)}\|v\|_{H^s(\cS)}.
\end{equation*}
\end{prop}

\begin{proof}
The proof runs similarly to that of Theorem \ref{prop-products}.
All what is needed is an extension of \eqref{ineq-weighted_l1} to the
case of two independent variables $l$ and $m$.
Indeed, we simply have that the inequality
\begin{equation*}
 \sum_{l=0}^\infty\sum_{m=-l}^l\sqrt{2l+1}|\,\widehat{u}_l^m|\leq
C\|u\|_{H^s(\cS)}
\end{equation*}
holds for $u\in H^s(\cS)$ if $s>3/2$, as in the proof of Lemma
\ref{lem-weighted_l1}.
\end{proof}

\subsection{The square of the gradient of a function in $H^s_{\rm ax}(\cS)$}

In the case of axially symmetric functions, we have that
\begin{equation*}
 |\na w|^2=w_r^2+w_\te^2.
\end{equation*}
The following two lemmas will be decisive for the proof of existence for problem \eqref{backus-problem} of Section \ref{sec:backus-problem}.

\begin{lem}
\label{cor-r_products}
Let $u,v$ be harmonic functions in $\Om$, continuous up to the boundary $\cS$, and such that $u, v\to 0$ as $|x|\to\infty$. If $u,v\in H^{s+1}_{\rm ax}(\cS)$ for $s>1$, 
then $u_r v_r\in H^s_{\rm ax}(\cS)$ and
\begin{equation*}
 \|u_rv_r\|_{H^s(\cS)}\leq C\|u\|_{H^{s+1}(\cS)}\|v\|_{H^{s+1}(\cS)}
\end{equation*}
for some $C>0$ independent of $u,v$. 
\end{lem}

\begin{proof}
If $u=\sum\limits_{l=0}^\infty r^{-l-1}\widehat{u}_l\,Y_l^0$, then $u_r=-\sum\limits_{l=0}^\infty r^{-l-2} (l+1)\widehat{u}_l\,Y_l^0$, and similarly for $v$. 
Hence, Theorem \ref{prop-products} gives
\begin{align*}
 \|u_rv_r\|_{H^s(\cS)}&\leq C\|u_r\|_{H^s(\cS)}\|v_r\|_{H^s(\cS)}\\
 &=C\sqrt{\sum_{l=0}^\infty(l+1)^{2s+2}|\widehat{u}_l|^2}\,\sqrt{\sum_{l=0}^\infty(l+1)^{2s+2}|\widehat{v}_l|^2}\\
 &=C\|u\|_{H^{s+1}(\cS)}\|v\|_{H^{s+1}(\cS)}, 
\end{align*}
as desired. 
\end{proof}

We now turn to the estimate for $v_\te^2$ in the $H^s(\cS)$-norm. 
Unlike the case of $v_r$, we do not have a simple expression of the spherical harmonics expansion of $v_\te$. 
Nevertheless, we will show that there is one for $v_\te^2$, by using the fact that the Legendre polynomial $P_l^0(z)$ of degree $l$, which we will denote here by $P_l(z)$, solves the differential equation:
\begin{equation}
\label{legendre_eqn}
 \left[\left(1-z^2\right)P_l'\right]'+l(l+1)P_l=0.
\end{equation}

\smallskip

\begin{lem}
\label{lem-theta_products}
Let $u,v$ be harmonic functions in $\Om$, continuous up to the boundary $\cS$, and such that $u, v\to 0$ as $|x|\to\infty$. If $u,v\in H^{s+1}_{\rm ax}(\cS)$ for $s>1$,
then $u_\te v_\te\in H^s_{\rm ax}(\cS)$ and
\begin{equation*}
 \|u_\te v_\te\|_{H^s(\cS)}\leq C\|u\|_{H^{s+1}(\cS)}\|v\|_{H^{s+1}(\cS)}
\end{equation*}
for some $C>0$ independent of $u,v$. 
\end{lem}

\begin{proof}
Recall that 
\begin{equation*}
 Y_l^0(\theta,\varphi)=\sqrt{\frac{2l+1}{4\pi}}P_l(\sin\theta). 
\end{equation*}
Hence, if $u(\theta)=\sum\limits_{i=0}^\infty \widehat{u}_iY_i^0(\theta,\varphi)$ and $v(\theta)=\sum\limits_{j=0}^\infty  \widehat{v}_jY_j^0(\theta,\varphi)$, we infer that
\begin{align*}
 u_\te(\te) v_\te(\te)
 &=\sum_{i=0}^\infty\sum_{j=0}^\infty\widehat{u}_i\widehat{v}_j\sqrt{\frac{(2i+1)(2j+1)}{(4\pi)^2}}\cos^2\te\, P_i'(\sin\te)\,P_j'(\sin\te)\\
 &= \sum_{l=0}^\infty c_lY_l^0(\te,\fhi), 
\end{align*}
where
\begin{align*}
 c_l&=2\pi\int_{-\frac{\pi}{2}}^{\frac{\pi}{2}} u_\te(\te) v_\te(\te) Y_l^0(\te,\fhi)\,\cos\te\,d\te \\
 &=2\pi\sum_{i=0}^\infty\sum_{j=0}^\infty\widehat{u}_i\widehat{v}_j\sqrt{\frac{(2i+1)(2j+1)(2l+1)}{(4\pi)^3}}\int_{-1}^1(1-z^2)P_i'(z)P_j'(z)P_l(z)\,dz. 
\end{align*}
The last integral can be represented by the symbols $W_{0,0,0}^{i,j,l}$ defined in \eqref{w_ijl}. 
Indeed, \eqref{legendre_eqn} gives that
\begin{equation*}
 \left\{\left(1-z^2\right)\left[P_i'(z)P_l(z)-P_i(z)P_l'(z)\right]\right]\}'P_j=\left[l(l+1)-i(i+1)\right]P_i(z)P_j(z)P_l(z). 
\end{equation*}
An  integration by parts then gives:
\begin{align*}
 \left[i(i+1)-l(l+1)\right]\int_{-1}^1P_iP_jP_l\,dz
 &=\int_{-1}^1\left(1-z^2\right)\left(P_i'P_l-P_iP_l'\right)P_j'\,dz\\
 &\hspace{-2cm}=2\int_{-1}^1\left(1-z^2\right)P_i'P_j'P_l\,dz-j(j+1)\int_{-1}^1P_iP_jP_l\,dz. 
\end{align*}
Hence, we compute:
\begin{align*}
 c_l&=\sum_{i, j=0}^\infty\widehat{u}_i\widehat{v}_j\left[i(i\!+\!1)+j(j\!+\!1)\!-\!l(l\!+\!1)\right]
\sqrt{\frac{(2i\!+\!1)(2j\!+\!1)(2l\!+\!1)}{64\,\pi}}\int_{-1}^1\!P_iP_jP_l\,dz\\
 &=\sum_{i, j=0}^\infty\frac{i(i\!+\!1)+j(j\!+\!1)\!-\!l(l\!+\!1)}{2}\sqrt{(2i\!+\!1)(2j\!+\!1)(2l\!+\!1)} \, \,W_{0,0,0}^{i,j,l}\,\widehat{u}_i\,\widehat{v}_j. 
\end{align*}
Note that $W_{0,0,0}^{i,j,l}$ is nonzero only when $|i-j|\leq l\leq i+j$, so that we have:
\begin{equation*}
 -ij\leq\frac{i(i+1)+j(j+1)-l(l+1)}{2}\leq ij. 
\end{equation*}
Therefore, as in the proof of Theorem \ref{prop-products}, we can split up the sum into two summands and obtain the inequality:
\begin{align*}
 (l+1)^s|c_l|&\leq\sum_{i=0}^\infty\sum_{j=i}^\infty(2i+1)^{3/2}(2j+1)^{s+1}(l+1)|W^{i,j,l}_{0,0,0}||\widehat{u}_{i}||\widehat{v}_{j}|\\
 &\quad +\sum_{j=0}^\infty\sum_{i=j}^\infty(2i+1)^{s+1}(2j+1)^{3/2}(l+1)|W^{i,j,l}_{0,0,0}||\widehat{u}_{i}||\widehat{v}_{j}|. 
\end{align*}
The desired estimate then follows thanks to the same arguments as those in the proof of Theorem \ref{prop-products}. 
\end{proof}

\smallskip

\section{Series solution for the linearized problem}
\label{sec:series-solution}


In this section, we collect the results on the linearized problem \eqref{oblique}--\eqref{dirichlet-equator}, which will be instrumental for the proof of our main theorem in Section \ref{sec:backus-problem}.
\par
In the spherical system of coordinates, we can compute that
$$
d=\frac{\sin\te}{r^2} \ \mbox{ and } \ \na u\cdot\na v=u_r \,v_r+\frac{u_\te\, v_\te}{r^2} +\frac{u_\fhi\, v_\fhi}{r^2 \cos^2\te}.
$$
Thus, problem \eqref{oblique}--\eqref{dirichlet-equator} reads as
\begin{subequations}
\label{oblique-spherical}
\begin{eqnarray}
&&\frac{1}{r^2}(r^2 v_r)_r +\frac{1}{r^2\cos \te}(\cos \te \, v_\te)_\te 
+\frac{1}{r^2\cos^2 \te}\,v_{\fhi \fhi}=0 
\label{oblique-spherical-eq} 
\\
&&\qquad\qquad\qquad\qquad\qquad \mbox{ for } \ r\ge 1, \ -\pi/2\le\te\le\pi/2, \ -\pi\le\fhi<\pi,
\nonumber
\\
\label{oblique-spherical-obc}
&&-2\,\sin\te\,v_r+\cos\te\,v_\te=\frac12\,f \ \mbox{ for } \ r=1, \ -\pi/2\le\te\le\pi/2, \ -\pi\le\fhi<\pi,
\\
\label{oblique-spherical-c}
&&v\to 0 \ 
\ \mbox{ uniformly in } \ -\pi/2\le\te\le\pi/2, \ -\pi\le\fhi<\pi \ 
\mbox{ as } \ r\to\infty,
\\
\label{oblique-spherical-dbcai}
&&
\label{oblique-spherical-d}
v=h \ \mbox{ for } \ r=1, \ \te=0,  \ -\pi\le\fhi<\pi.
\end{eqnarray} 
\end{subequations}

The main result of this section is the following.

\begin{thm}[Unique existence and fractional Sobolev estimates]
\label{thm:oblique-solution}
Suppose that $f \in H^s(\cS)$, with $f_\fhi \in H^{s-1/2}(\cS)$,
and $h \in H^{s+3/4}(\cE)$ for some $s>1$.
Then \eqref{oblique-spherical} has a unique classical solution 
$v \in C^2(\Om) \cap C^1(\overline{\Om})$.
Furthermore, the solution satisfies $v|_{\cS} \in H^{s+1}(\cS)$ and
\begin{equation}
\label{estimate-solution-leq}
\|v\|_{H^{s+1}(\cS)} \le C\left( \| f\|_{H^s(\cS)} +\| f_\fhi\|_{H^{s-1/2}(\cS)}
+\| h\|_{H^{s+3/4}(\cE)}\right),
\end{equation}
where $C>0$ is a constant independent of $f$ and $h$.
\par
In particular, if $f \in H^s_{\rm ax}(\cS)$ and $h$ is a constant,
then $v|_{\cS} \in H^{s+1}_{\rm ax}(\cS)$ and
\begin{equation*}
\|v\|_{H^{s+1}(\cS)} \le C\left( \| f\|_{H^s(\cS)} +|h|\right).
\end{equation*}
\end{thm}


The proof of Theorem~\ref{thm:oblique-solution} is given in Section~\ref{sec:pf_thm31}.

\subsection{Formal derivation of a series solution}\label{sec:representation-formula}

We start by formally deriving a representation formula 
of a solution $v$ of problem \eqref{oblique-spherical}
(computations here will be verified in Proposition~\ref{lem:series-expansion} below).
The formula is given by
\begin{equation}
\label{v-representation-formula}
v(r,\te,\fhi)=\sum_{l=0}^\infty \sum_{m=-l}^l (b_l^m +c_l^m) r^{-l-1} Y_l^m(\te,\fhi),
\end{equation}
where $b_l^m$ and $c_l^m$ are defined as follows.
We set
\begin{align*}
&\be_l^m=\sqrt{\frac{(l-|m|)(l+|m|)}{(2l-1)(2l+1)}}\, , 
\\ 
&\ga_l^m=-\frac{(l+1)\,\be_l^m}{3(l+2)\,\be_{l+1}^m}
 =-\frac{l+1}{3(l+2)} \sqrt{\frac{(2l+3)(l-|m|)(l+|m|)}{(2l-1)(l+1-|m|)(l+1+|m|)}}\, 
\\
&\mbox{for } \ m=0,\pm 1,\dots, \ l=|m|,|m|+1,\dots,
\end{align*}
and put
\begin{equation*}
\Ga_0^m=1,
\qquad
\Ga_k^m=\prod_{j=1}^k \ga_{|m|+2j-1}^m
\ \mbox{ for } \ m=0,\pm 1,\dots, \ k=1, 2, \dots.
\end{equation*}
For $f \in L^2(\cS)$ and $h \in L^2(\cE)$,
we write
$$
a_l^m=\frac{1}{6(l+2)\be_{l+1}^m}\,\widehat{f}_l^m,
\quad
\tilde a^m=\frac{\widehat{h} _m-\sum\limits_{k=0}^\infty b_{|m|+2k}^m \al_{|m|+2k}^m P_{|m|+2k}^{|m|} (0)}{\sum\limits_{k=0}^\infty \Ga_k^m \al_{|m|+2k}^m P_{|m|+2k}^{|m|} (0)}
$$
(note that the denominator in the definition of $\tilde a^m$ is nonzero
due to the inequality \eqref{Gamma_alpha_P_bound} 
of Lemma~\ref{lem:Gamma-estimates} below).
$b_l^m$ is then defined by the recurrence relation
\begin{alignat*}{2}
&b_{|m|-1}^m=b_{|m|}^m=0
&\quad &\mbox{ for } \ m=0,\pm 1,\dots,
\\
&b_{l+1}^m=\ga_l^m b_{l-1}^m+a_l^m
&\quad &\mbox{ for } \ m=0,\pm 1,\dots, \ l=|m|,|m|+1,\dots,
\end{alignat*}
and $c_l^m$ is given by
$$
c_{|m|+2k-1}^m=0,
\quad
c_{|m|+2k}^m=\Ga_k^m \,\tilde a^m
\ \mbox{ for } \ m=0,\pm 1,\dots, \ k=0,1,\dots.
$$

Let us derive \eqref{v-representation-formula}.
Let $\widehat{v}_l^m(r)$ denote the Fourier-Laplace coefficients of $v(r,\cdot,\cdot)$, that is, 
$$
\widehat{v}_l^m(r)
=\left\lan v(r,\cdot,\cdot), Y_l^m \right\ran_{L^2(\cS)}
=\int_{-\pi}^\pi \int_{-\frac{\pi}{2}}^{\frac{\pi}{2}} 
v(r,\te,\fhi) \, \overline{Y_l^m(\te,\fhi)}\cos \te\, d\te d\fhi.
$$
For abbreviation, we write $\widehat{v}_l^m$ instead of $\widehat{v}_l^m(1)$ 
and, for convenience, we set $\widehat{v}_l^{m}=0$ if $l<|m|$.
First, we observe that $v$ has the form
\begin{equation}
\label{v-expansion}
v(r,\te,\fhi)=\sum_{l=0}^\infty \sum_{m=-l}^l \widehat{v}_l^m r^{-l-1} Y_l^m(\te,\fhi).
\end{equation}
This follows from \eqref{oblique-spherical-eq} and \eqref{oblique-spherical-c}.
Indeed, multiplying \eqref{oblique-spherical-eq} by $\overline{Y_l^m}$, 
integrating over $\cS$ and using the fact that $-\Delta_{\cS} Y_l^m=l(l+1)Y_l^m$,
we see that $\widehat{v}_l^m(r)$ satisfies 
$$
\frac{d^2\widehat{v}_l^m}{dr^2}(r)+\frac{2}{r}\frac{d\widehat{v}_l^m}{dr}(r) 
-\frac{l(l+1)}{r^2}\widehat{v}_l^m(r)=0
\quad \mbox{for } r>1.
$$
This together with the condition \eqref{oblique-spherical-c}
gives $\widehat{v}_l^m(r)=\widehat{v}_l^m r^{-l-1}$,
and hence we obtain \eqref{v-expansion}.
\par
Next, we derive a recurrence relation for $\widehat{v}_l^m$ from \eqref{oblique-spherical-obc}.
It is known that the following recurrence relations hold
(see \cite[(15.88)]{AWH} for the first equality 
and \cite[(15.88), (15.89), (15.92)]{AWH} for the second equality):
\begin{gather*}
zP_l^{|m|} (z)
=\frac{l-|m|+1}{2l+1}P_{l+1}^{|m|} (z) +\frac{l+|m|}{2l+1}P_{l-1}^{|m|} (z),
\\
(1-z^2)\frac{d P_l^{|m|}}{dz} (z)
=-\frac{l(l-|m|+1)}{2l+1}P_{l+1}^{|m|} (z) +\frac{(l+1)(l+|m|)}{2l+1}P_{l-1}^{|m|} (z).
\end{gather*}
Hence we have that
\begin{gather*}
\sin\te\, Y_l^m=\be_{l+1}^mY_{l+1}^m+\be_l^mY_{l-1}^m,
\\
\cos\te\, \frac{\pa Y_l^m}{\pa \te}=-l\be_{l+1}^mY_{l+1}^m+(l+1)\be_l^mY_{l-1}^m.
\end{gather*}
Here, $Y_l^m=0$ if $l<|m|$.
From these identities, we see that
\begin{align*}
\lan \sin\te\,v_r,Y_l^m \ran_{L^2(\cS)}
&=\int_{-\pi}^\pi \int_{-\frac{\pi}{2}}^{\frac{\pi}{2}} 
v_r(r,\te,\fhi) \, \sin \te\, \overline{Y_l^m(\te,\fhi)} \cos \te\, d\te d\fhi
\\
&=\frac{d}{dr} \left( \be_{l+1}^m \widehat{v}_{l+1}^m(r)
+\be_l^m \widehat{v}_{l-1}^m(r) \right)
\\
&=-(l+2)\be_{l+1}^m \widehat{v}_{l+1}^m r^{-l-3} -l \be_l^m \widehat{v}_{l-1}^m r^{-l-1},
\end{align*}
and
\begin{align*}
&\lan \cos\te\,v_\te,Y_l^m \ran_{L^2(\cS)}
\\
&=\int_{-\pi}^\pi \int_{-\frac{\pi}{2}}^{\frac{\pi}{2}} 
v_\te(r,\te,\fhi) \, \overline{Y_l^m(\te,\fhi)} \cos^2 \te\, d\te d\fhi
\\
&=\int_{-\pi}^\pi \int_{-\frac{\pi}{2}}^{\frac{\pi}{2}} 
v(r,\te,\fhi) \, \left( \cos \te\, \overline{\frac{\pa Y_l^m}{\pa \te}(\te,\fhi)}
-2\sin \te \overline{Y_l^m(\te,\fhi)} \right) \cos \te\, d\te d\fhi
\\
&=(l+2)\be_{l+1}^m \widehat{v}_{l+1}^m r^{-l-2} -(l-1) \be_l^m \widehat{v}_{l-1}^m r^{-l},
\end{align*}
where we used integration by parts.
Thus, multiplying \eqref{oblique-spherical-obc} by $\overline{Y_l^m}$ 
and integrating over $\cS$, 
we find the recurrence relation
\begin{equation*}
\widehat{v}_{l+1}^m=\ga_l^m \widehat{v}_{l-1}^m+a_l^m
\ \mbox{ for } \ m=0,\pm 1,\dots, \ l=|m|,|m|+1,\dots.
\label{clm-recurrence-relation}
\end{equation*}

Finally, we consider the condition \eqref{oblique-spherical-d}.
Interchanging the sum in \eqref{v-expansion}, 
we have that
\begin{equation*}
v(1,\te,\fhi)=\sum_{m=-\infty}^\infty \widehat{v}_m(\te) e^{im\fhi},
\quad \mbox{where} \quad 
\widehat{v}_m(\te)=\sum_{l=|m|}^\infty \widehat{v}_l^m \al_l^m P_l^m(\sin \te).
\label{v-expansion2}
\end{equation*}
Hence, by \eqref{oblique-spherical-d}, we deduce that
\begin{equation}
\label{oblique-spherical-d2}
\widehat{h}_m=\widehat{v}_m(0)=\sum_{l=|m|}^\infty \widehat{v}_l^m \al_l^m P_l^m(0).
\end{equation}
The recurrence relations for $b_l^m$ and $\widehat{v}_l^m$ 
show that $d_l^m=\widehat{v}_l^m-b_l^m$ satisfies
\begin{alignat*}{2}
&d_{|m|-1}^m=0, \ d_{|m|}^m=\widehat{v}_{|m|}^m
&\quad &\mbox{ for } \ m=0,\pm 1,\dots,
\\
&d_{l+1}^m=\ga_l^m d_{l-1}^m
&\quad &\mbox{ for } \ m=0,\pm 1,\dots, \ l=|m|,|m|+1,\dots,
\end{alignat*}
and therefore $d_{|m|+2k-1}^m=0$ and $d_{|m|+2k}^m=\Ga_k^m \,\widehat{v\,}_{|m|}^m$.
Plugging $\widehat{v}_l^m=b_l^m+d_l^m$ into \eqref{oblique-spherical-d2}
and using the fact that $P_l^{|m|} (0)=0$ if $l-|m|$ is odd (see \cite[(15.96)]{AWH}),
we find that
\begin{equation*}
\widehat{h}_m=\sum_{k=0}^\infty b_{|m|+2k}^m \al_{|m|+2k}^m P_{|m|+2k}^{|m|} (0)
+\widehat{v}_{|m|}^m \sum_{k=0}^\infty \Ga_k^m \al_{|m|+2k}^m P_{|m|+2k}^{|m|} (0).
\end{equation*}
This gives $d_l^m=c_l^m$,
and thus we obtain \eqref{v-representation-formula}.

\begin{rem}\label{rem:v-symmetry}
Suppose that $f$ and $h$ are independent of $\fhi$,
that is, $\widehat{f}_l^m=\widehat{h}_m=0$ if $m \neq 0$.
Then, by definition,
we see that $b_l^m=c_l^m=0$ unless $m=0$.
This shows that the function $v$ defined by \eqref{v-representation-formula} 
is also independent of $\fhi$.
\end{rem}

\subsection{Proof of Theorem~\ref{thm:oblique-solution}}\label{sec:pf_thm31}

We divide the proof of Theorem~\ref{thm:oblique-solution} into a sequence of lemmas.
First we show some estimates to be mainly used in deriving \eqref{estimate-solution-leq}.

\begin{lem}\label{lem:Gamma-estimates}
The inequalities
\begin{gather}
c\,\root 4\of{\frac{|m|+2k+1}{2k+1}} 
\le (-1)^{\frac{|m|-m}{2}+k} \al_{|m|+2k}^m P_{|m|+2k}^{|m|} (0) 
\le C\,\root 4 \of{\frac{|m|+2k+1}{2k+1}},
\label{alpha_P_bounds}
\\
c\,\root 4\of{\frac{|m|+1}{(2k+1)(|m|+2k+1)}}
\le (-3)^k\Ga_k^m
\le C\,\root 4 \of{\frac{|m|+1}{(2k+1)(|m|+2k+1)}},
\label{Gamma_bounds}
\\
c\,\root 4\of{|m|+1}
\le \left| \sum_{k=0}^\infty \Ga_k^m \al_{|m|+2k}^m P_{|m|+2k}^{|m|} (0)\right|
\le C\,\root 4\of{|m|+1}
\label{Gamma_alpha_P_bound}
\end{gather}
hold with some positive constants $c$ and $C$.
\end{lem}
\begin{proof}
It is known (see \cite[(15.96)]{AWH}) that
\begin{align*}
P_{|m|+2k}^{|m|} (0)=(-1)^{|m|+k}\frac{(2|m|+2k-1)!!}{(2k)!!},
\end{align*}
which gives:
\begin{align*}
&(-1)^{\frac{|m|-m}{2}+k} \al_{|m|+2k}^m P_{|m|+2k}^{|m|} (0)
\\
&=\frac{\sqrt{(2k)!}}{(2k)!!} \frac{(2|m|+2k+1)!!}{\sqrt(2|m|+2k+1)!}\, \sqrt{\frac{2|m|+4k+1}{4\pi(2|m|+2k+1)}}.
\end{align*}
Also, by definition, we have that
\begin{multline*}
 (-3)^k\Ga_k^m 
=\sqrt{\frac{2|m|+4k+1}{2|m|+1}}\,
\prod_{j=1}^k \frac{|m|+2j}{|m|+2j+1}\, \sqrt{\frac{2j-1}{2j} \frac{2|m|+2j-1}{2|m|+2j}}
\\
=\frac{(|m|\!+\!1)!! \,(|m|\!+\!2k)!!}{|m|!! \, (|m|\!+\!2k\!+\!1)!!}\,\sqrt{\frac{2|m|\!+\!4k\!+\!1}{2|m|\!+\!1}}
\sqrt{\frac{(2k\!-\!1)!!}{(2k)!!} \frac{(2|m|)!!}{(2|m|\!-\!1)!!} \frac{(2|m|\!+\!2k\!-\!1)!!}{(2|m|\!+\!2k)!!}}.
\end{multline*}
Therefore \eqref{alpha_P_bounds} and \eqref{Gamma_bounds} follow
from Lemma~\ref{factorial_inequalities} and some simple estimates.
The inequality \eqref{Gamma_alpha_P_bound} is derived easily 
by \eqref{alpha_P_bounds} and \eqref{Gamma_bounds}.
\end{proof}

Next, we derive an estimate of the $H^s(\cS)$-norm
of the formal solution \eqref{v-representation-formula}. 
For $a \in \RR$, we write $a_+=\max \{ a,0\}$.
\begin{lem}[Fractional Sobolev estimates for $v$]
\label{lem:fractional-Sobolev-estimate}
Suppose that $f \in H^s(\cS)$, $f_\fhi \in H^{(s-1/2)_+}(\cS)$ 
and $h \in H^{s+3/4}(\cE)$ for some $s \ge 0$.
Then the function $v$ defined by \eqref{v-representation-formula}
satisfies $v|_{\cS} \in H^{s+1}(\cS)$ and 
\begin{equation}
\label{estimate-solution-leq0}
\|v\|_{H^{s+1}(\cS)} \le C\left( \| f\|_{H^s(\cS)} +\| f_\fhi\|_{H^{(s-1/2)_+}(\cS)}
+\| h\|_{H^{s+3/4}(\cE)}\right)
\end{equation}
for some constant $C>0$ independent of $f$ and $h$.
\end{lem}

\begin{proof}
Throughout the proof,
$c$ and $C$ denote generic positive constants depending only on $s$, which may change from formula to formula.
\par
Notice that, since $(Y_l^m)_\fhi=im Y_l^m$, we have that
$
\widehat{(f_\fhi)}_l^m =im \widehat{f}_l^m,
$
and hence
\begin{equation}
\label{norm-derivative-f}
\| f_\fhi\|_{H^s(\cS)}^2=\sum_{l=0}^\infty \sum_{m=-l}^l m^2 (l+1)^{2s} |\widehat f_l^m|^2.
\end{equation}
\par
Next, 
we write $v|_{\cS}$ as
\begin{equation*}
v|_{\cS}=v_1+v_2, 
\quad
v_1=\sum_{l=0}^\infty \sum_{m=-l}^l b_l^m Y_l^m,
\quad
v_2=\sum_{l=0}^\infty \sum_{m=-l}^l c_l^m Y_l^m,
\end{equation*}
and estimate the norms of $v_1$ and $v_2$ separately.
We first consider $v_1$.
It is elementary to show that $|\ga_l^m| \le 2/3$,
and hence $|b_{l+1}^m| \le 2|b_{l-1}^m|/3 +|a_l^m|$.
Thus, applying Lemma~\ref{lem:recurrence-inequality} 
with $p_k=|b_{|m|+2k}^m|$, $q_k=|a_{|m|+2k+1}^m|$,
$\si=2/3$, $\tau_1=2(s+1)$, $\tau_2=0$, and $\chi=|m|$
gives:
$$
\sum_{k=0}^\infty (|m|+2k+1)^{2(s+1)} \left| b_{|m|+2k}^m\right|^2
\le C\sum_{k=0}^\infty (|m|+2k+1)^{2(s+1)} \left|a_{|m|+2k+1}^m\right|^2.
$$
We use Lemma~\ref{lem:recurrence-inequality} again 
with $p_k=|b_{|m|+2k-1}^m|$ and $q_k=|a_{|m|+2k}^m|$,
and combine the resulting inequality with the above inequality to obtain that
$$
\sum_{l=|m|}^\infty (l+1)^{2(s+1)} |b_l^m|^2
\le C\sum_{l=|m|}^\infty (l+1)^{2(s+1)} |a_l^m|^2.
$$
Since
\begin{align*}
(l+1)^{2(s+1)} |a_l^m|^2
&\le \frac{(l+1)^{2(s+1)}}{(l+1-|m|)(l+1+|m|)} |\widehat{f}_l^m|^2
\\
&=(l+1)^{2s} |\widehat{f}_l^m|^2 +\frac{m^2 (l+1)^{2s}}{(l+1-|m|)(l+1+|m|)} |\widehat{f}_l^m|^2
\\
&\le (l+1)^{2s} |\widehat{f}_l^m|^2
+m^2 (l+1)^{(2s-1)_+} |\widehat{f}_l^m|^2,
\end{align*}
we find
$$
\sum_{l=|m|}^\infty (l+1)^{2(s+1)} |b_l^m|^2
\le  C\sum_{l=|m|}^\infty (l+1)^{2s} |\widehat{f}_l^m|^2
+C\sum_{l=|m|}^\infty m^2 (l+1)^{(2s-1)_+} |\widehat{f}_l^m|^2.
$$
By taking the sum over $m$ and applying \eqref{norm-derivative-f},
we conclude that
\begin{equation}\label{v1-estimate}
\| v_1\|_{H^{s+1}(\cS)} \le C\left( \| f\|_{H^s(\cS)} +\|  f_\fhi\|_{H^{(s-1/2)_+}(\cS)}\right).
\end{equation}
\par
In order to examine $v_2$, we estimate $\tilde a^m$.
We see from \eqref{alpha_P_bounds} and the Cauchy-Schwarz inequality that
\begin{align*}
&\left| \sum_{k=0}^\infty b_{|m|+2k}^m \al_{|m|+2k}^m P_{|m|+2k}^{|m|} (0)\right|
\\
&\le C \sum_{k=0}^\infty \root 4 \of{\frac{|m|+2k+1}{2k+1}} 
\left|b_{|m|+2k}^m \right|
\\
&\le C \sqrt{\sum_{k=0}^\infty \frac{1}{(2k+1)^{3/2}}}
\sqrt{\sum_{k=0}^\infty (2k+1)  \sqrt{|m|+2k+1} \left| b_{|m|+2k}^m \right|^2 }.
\end{align*}
Applying Lemma~\ref{lem:recurrence-inequality}
with $p_k=|b_{|m|+2k}^m|$, $q_k=|a_{|m|+2k+1}^m|$,
$\si=2/3$, $\tau_1=1/2$, $\tau_2=1$, and $\chi=|m|$ yields that
\begin{align*}
\sum_{l=0}^\infty  (2k\!+\!1) \sqrt{|m|\!+\!2k\!+\!1}\left| b_{|m|\!+\!2k}^m \right|^2
&\le C\sum_{k=0}^\infty  (2k+1) \sqrt{|m|+2k+1} \left|a_{|m|+2k+1}^m \right|^2
\\
&\le C\sum_{l=|m|}^\infty \sqrt{l+1}\, (l+1-|m|) |a_l^m|^2 
\\
&\le C\sum_{l=|m|}^\infty \frac{1}{\sqrt{l+1}}\, |\widehat{f}_l^m|^2.
\end{align*}
Hence
\begin{equation*}
\left| \sum_{k=0}^\infty b_{|m|+2k}^m \al_{|m|+2k}^m P_{|m|+2k}^{|m|} (0)\right|^2
\le C\sum_{l=|m|}^\infty \frac{1}{(l+1)^{1/2}} |\widehat{f}_l^m|^2.
\end{equation*}
From this and \eqref{Gamma_alpha_P_bound}, 
we obtain that
\begin{equation}\label{cm_bound}
|\tilde a^m|^2 \le \frac{C}{\sqrt{|m|+1}} \left[ |\widehat{h}_m|^2
+\sum_{l=|m|}^\infty \frac{1}{\sqrt{l+1}}\, |\widehat{f}_l^m|^2 \right].
\end{equation}
\par
Now, we are ready to estimate $v_2$.
Note that
\begin{align*}
\| v_2\|_{H^{s+1}(\cS)}^2 
&=\sum_{l=0}^\infty \sum_{m=-l}^l (l+1)^{2(s+1)} |c_l^m|^2
\\
&=\sum_{m=-\infty}^\infty \left[ \sum_{k=0}^\infty
(|m|+2k+1)^{2(s+1)} (\Ga_k^m)^2 \right] |\tilde a^m|^2.
\end{align*}
By \eqref{Gamma_bounds} and the fact that $|m|+1 \le |m|+2k+1 \le (|m|+1)(2k+1)$,
we have
\begin{align*}
\sum_{k=0}^\infty (|m|+2k+1)^{2(s+1)} (\Ga_k^m)^2
&\le C\sum_{k=0}^\infty \frac{[(|m|+1)(2k+1)]^{2(s+1)}}{3^{2k}} 
\\
&\le C(|m|+1)^{2(s+1)}.
\end{align*}
This, together with \eqref{cm_bound}, gives
\begin{equation*}
\| v_2\|_{H^{s+1}(\cS)}^2
\le C\sum_{m=-\infty}^\infty (|m|+1)^{2s+\frac{3}{2}} |\widehat{h}_m|^2
+C\sum_{m=-\infty}^\infty \sum_{l=|m|}^\infty \frac{(|m|+1)^{2s+3/2}}{\sqrt{l+1}} |\widehat{f}_l^m|^2.
\end{equation*}
Since
$$
(|m|+1)^{2s+\frac{3}{2}}
\le 2(m^2+1) (l+1)^{2s-\frac{1}{2}}
\le 2m^2(l+1)^{(2s-1)_+ +\frac{1}{2}} +2(l+1)^{2s+\frac{1}{2}},
$$
we deduce that
\begin{align*}
&\| v_2\|_{H^{s+1}(\cS)}^2\\
&\le C\sum_{m=-\infty}^\infty (|m|\!+\!1)^{2s+\frac{3}{2}} |\widehat{h}_m|^2
+C\sum_{m=-\infty}^\infty \sum_{l=|m|}^\infty \left[ m^2(l\!+\!1)^{(2s-1)_+} +(l\!+\!1)^{2s} \right] |\widehat{f}_l^m|^2\\
&=C\left( \| h\|_{H^{s+3/4}(\cE)} +\|  f_\fhi\|_{H^{(s-1/2)_+}(\cS)}
+\| f\|_{H^s(\cS)}\right).
\end{align*}
Combining this and \eqref{v1-estimate}, 
we obtain \eqref{estimate-solution-leq0}.
Thus the proof is completed.
\end{proof}

In the next lemma we check that, 
under the assumptions of Theorem~\ref{thm:oblique-solution},
the function $v$ given by \eqref{v-representation-formula} 
is indeed a classical solution of \eqref{oblique-spherical}.

\begin{lem}[Regularity of $v$]
\label{lem:regularity-series-solution}
Suppose that $f \in H^s(\cS)$, $f_\fhi \in H^{s-1/2}(\cS)$ 
and $h \in H^{s+3/4}(\cE)$ for some $s>1$,
and let $v$ be defined by \eqref{v-representation-formula}.
Then $v$ belongs to $C^2 (\Om) \cap C^1(\overline{\Om})$ 
and satisfies \eqref{oblique-spherical} in the classical sense.
\end{lem}

\begin{proof}
We know that the Fourier-Laplace coefficients $\widehat{v}_l^m$ of $v|_{\cS}$
is given by $\widehat{v}_l^m=b_l^m+c_l^m$.
From \eqref{Ylm-identity} and the Cauchy-Schwarz inequality, we have that
\begin{align}
\sum_{l=0}^\infty \sum_{m=-l}^l |\widehat{v}_l^m r^{-l-1} Y_l^m(\te,\fhi)|
&\le r^{-1}\sum_{l=0}^\infty \sum_{m=-l}^l |\widehat{v}_l^m Y_l^m(\te,\fhi)|
\nonumber
\\
&\le r^{-1}\|v\|_{H^{s+1}(\cS)}
\sqrt{\sum_{l=0}^\infty\sum_{m=-l}^l(l+1)^{-2(s+1)}|Y_l^m(\te,\fhi)|^2}
\nonumber
\\
&\le r^{-1}\|v\|_{H^{s+1}(\cS)} 
\sqrt{\frac{1}{2\pi}\sum_{l=0}^\infty(l+1)^{-1-2s}}.
\label{pointwise-estimate-v}
\end{align}
Lemma~\ref{lem:fractional-Sobolev-estimate} and the fact that $s>0$
show that the right-hand side is bounded by a constant independent of $(r,\te,\fhi)$.
Hence the series on the left-hand side converges uniformly on $\overline{\Om}$.
Since it is well-known that a uniform limit of a sequence of harmonic functions is smooth and harmonic,
we deduce that $v$ is of class $C^2$ in $\Om$ and satisfies \eqref{oblique-spherical-eq}.
In addition, from \eqref{pointwise-estimate-v}, 
we see at once that \eqref{oblique-spherical-c} is satisfied.

Next, we show that $v \in C^1(\overline{\Om})$ 
and that \eqref{oblique-spherical-obc} and \eqref{oblique-spherical-d} hold..
It suffices to show that the series 
\begin{equation}
\label{v-expansion-derivatives}
\begin{gathered}
\sum_{l=0}^\infty \sum_{m=-l}^l (l+1)\widehat{v}_l^m r^{-l-2} Y_l^m(\te,\fhi)
\left( =\sum_{l=0}^\infty \sum_{m=-l}^l \frac{\pa}{\pa r}\left( \widehat{v}_l^m r^{-l-2} Y_l^m(\te,\fhi)\right) \right),
\\
\sum_{l=0}^\infty \sum_{m=-l}^l \widehat{v}_l^m r^{-l-1} \frac{\pa Y_l^m}{\pa \te}(\te,\fhi),
\qquad
\sum_{l=0}^\infty \sum_{m=-l}^l \widehat{v}_l^m r^{-l-1} \frac{1}{\cos \te} \frac{\pa Y_l^m}{\pa \fhi}(\te,\fhi)
\end{gathered}
\end{equation}
are uniformly convergent on $\overline{\Om}$.
A computation similar to that in the derivation of \eqref{pointwise-estimate-v} gives
\begin{equation*}
\sum_{l=0}^\infty \sum_{m=-l}^l |(l+1)\widehat{v}_l^m r^{-l-2} Y_l^m(\te,\fhi)|
\le \|v\|_{H^{s+1}(\cS)} \sqrt{\frac{1}{2\pi}\sum_{l=0}^\infty(l+1)^{1-2s}}.
\end{equation*}
Since the right-hand side is finite if $s>1$,
we deduce that the first series of \eqref{v-expansion-derivatives} converges uniformly.
For the second and third series,
we use the identities
\begin{equation*}
\sum_{m=-l}^l \left| \frac{\pa Y_l^m}{\pa \te}(\theta,\varphi) \right|^2
=\sum_{m=-l}^l \frac{1}{\cos^2 \te} \left| \frac{\pa Y_l^m}{\pa \fhi}(\theta,\varphi) \right|^2
=\frac{l(l+1)(2l+1)}{8\pi},
\end{equation*}
which are derived by operating 
$\partial^2/\partial \theta_1 \partial \theta_2$ 
or $\partial^2/\partial \varphi_1 \partial \varphi_2$ 
to the equality \eqref{addition-theorem}
and then taking $(\theta_1,\varphi_1)=(\theta_2,\varphi_2)=(\theta,\varphi)$.
From these identities and the Cauchy-Schwarz inequality,
we have that
\begin{align*}
&\sum_{l=0}^\infty \sum_{m=-l}^l \left| \widehat{v}_l^m r^{-l-1} \frac{\pa Y_l^m}{\pa \te}(\te,\fhi)\right|
+\sum_{l=0}^\infty \sum_{m=-l}^l \left| \widehat{v}_l^m r^{-l-1} \frac{1}{\cos \te} \frac{\pa Y_l^m}{\pa \fhi}(\te,\fhi)\right|
\\
&\le \|v\|_{H^{s+1}(\cS)} 
\sqrt{\sum_{l=0}^\infty \sum_{m=-l}^l (l+1)^{-2(s+1)}
\left| \frac{\pa Y_l^m}{\pa \te}(\theta,\varphi) \right|^2}
\\
&\quad +|v\|_{H^{s+1}(\cS)} 
\sqrt{\sum_{l=0}^\infty \sum_{m=-l}^l (l+1)^{-2(s+1)}
\frac{1}{\cos^2 \te} \left| \frac{\pa Y_l^m}{\pa \fhi}(\theta,\varphi) \right|^2}
\\
&\le \|v\|_{H^{s+1}(\cS)} 
\sqrt{\frac{1}{\pi}\sum_{l=0}^\infty(l+1)^{1-2s}}.
\end{align*}
This shows that the second and third series of \eqref{v-expansion-derivatives} are uniformly convergent,
and thus the assertion follows.
\end{proof}

Finally, we prove the uniqueness of a solution of \eqref{oblique-spherical}.
\begin{lem}
\label{lem:uniqueness-oblique-spherical}
Let $v_1,v_2 \in C^2(\Om) \cap C^1(\overline{\Om})$ satisfy \eqref{oblique-spherical}.
Then $v_1=v_2$ on $\overline{\Om}$.
\end{lem}
\begin{proof}
Although the lemma can be shown in the same way as \cite[Theorem]{JM},
we give a proof for readers' convenience.

We know that the function $w=v_1-v_2$ satisfies
\begin{equation}
\label{w-eq}
\De w=0 \ \mbox{ in } \ \Om, 
\quad
\na d\cdot\na w=0 \ \mbox{ on } \ \cS,
\quad 
w\to 0 \ \mbox{ as } \ |x|\to\infty,
\quad
w=0 \ \mbox{ on } \ \cE.
\end{equation}
We show that $w \le 0$ on $\overline{\Om}$.
On the contrary, suppose that $w$ is positive somewhere.
Then, we can take a point $x_0 \in \overline{\Om}$ 
such that $w(x_0)=\sup_{\overline{\Om}} w>0$,
since $w$ decays at infinity.
By the last condition of \eqref{w-eq},
we have either $x_0 \in \Om$ or $x_0 \in \cS \setminus \cE$.
Assume $x_0 \in \Om$. 
Then, since $w$ is harmonic in $\Om$ and vanishes on $\cE$,
we see from the strong maximum principle that $w(x_0)=0$, a contradiction.
Assume $x_0 \in \cS \setminus \cE$.
In this case, we note that the tangential derivative of $w$ on $\cS$ vanishes at $x_0$.
From \eqref{gradient-d} and the second condition of \eqref{w-eq},
it follows that
\begin{equation*}
\left. \frac{\pa w}{\pa \nu}\right|_{x=x_0}
=\left. \frac{1}{2x_3}\left( \na d \cdot \na w -\tau \cdot \na w\right)\right|_{x=x_0}=0.
\end{equation*}
Hence the Hopf lemma and the last condition of \eqref{w-eq} give $w(x_0)=0$,
which is impossible.
Consequently $w$ is nonpositive everywhere.
The fact that $w \ge 0$ can be shown in the same way,
and therefore we obtain $w=0$.
\end{proof}

We can now prove Theorem~\ref{thm:oblique-solution}.

\begin{proof}[Proof of Theorem~\ref{thm:oblique-solution}]
The unique existence of a solution of \eqref{oblique-spherical} 
follows from Lemmas~\ref{lem:regularity-series-solution} 
and \ref{lem:uniqueness-oblique-spherical}.
The inequality \eqref{estimate-solution-leq} 
is a direct consequence of Lemma~\ref{lem:fractional-Sobolev-estimate}.
From Remark~\ref{rem:v-symmetry}, 
we see that the solution $v$ satisfies $v|_{\cS} \in H^{s+1}_{\rm ax}(\cS)$
if $f \in H^s_{\rm ax}(\cS)$ and $h$ is a constant.
Therefore the proof is completed.
\end{proof}

\subsection{Validity of the formula \eqref{v-representation-formula}}

We have proved in Theorem~\ref{thm:oblique-solution} 
that any classical solution of \eqref{oblique-spherical} 
is given by the formula \eqref{v-representation-formula},
provided that $f$ and $h$ are in certain Sobolev spaces.
At the end of this section, 
we prove that this is still true 
without assuming extra regularity conditions on $f$ and $h$.

\begin{prop}
\label{lem:series-expansion}
Any solution $v \in C^2(\Om) \cap C^1(\overline{\Om})$ 
of \eqref{oblique-spherical} is of the form \eqref{v-representation-formula}.
\end{prop}

\begin{proof}
Most of computations in Section~\ref{sec:representation-formula} are valid,
since the assumption $v \in C^2(\Om) \cap C^1(\overline{\Om})$ implies that
$\widehat{v}_l^m(r) \in C^2((1,\infty)) \cap C^1([1,\infty))$,
$f \in C(\cS)$ and $h \in C^1(\cE)$.
The only point where we have to verify is \eqref{oblique-spherical-d2}.
In order to ensure \eqref{oblique-spherical-d2},
we need to show that the equality \eqref{v-expansion2} holds in $L^2(\cE)$ for $\te=0$.
For this purpose,
we use the inequality (see \cite[Corollary~1]{L})
$$
\big| P_l^{|m|}(z)\big|
\le \sqrt[4]{\frac{64}{\pi^3 (2l+1)\sqrt{1-z^2}}} \sqrt{\frac{(l+|m|)!}{(l-|m|)!}},
$$
which gives 
$$
\big| \al_l^m P_l^{|m|} (\sin \te)\big|
\le \sqrt[4]{\frac{4(2l+1)}{\pi^5 \cos \te}}.
$$
From this and the Cauchy-Schwarz inequality,
we have that
\begin{align*}
|\widehat{v}_m(\te)|^2
&\le \sqrt{\frac{8}{\pi^5 \cos \te}}\left( \sum_{l=|m|}^\infty \sqrt[4]{l+1} |\widehat{v}_l^m| \right)^2
\\
&\le \sqrt{\frac{8}{\pi^5 \cos \te}}\left[ \sum_{l=0}^\infty \frac{1}{(l+1)^{3/2}} \right]
\left[ \sum_{l=|m|}^\infty (l+1)^2|\widehat{v}_l^m|^2\right].
\end{align*}
Since Proposition~\ref{prop:C-infty-in-Hs} 
and the assumption $v|_{\cS} \in C^1(\cS)$ show that
$$
\sum_{m=-\infty}^\infty \sum_{l=|m|}^\infty (l+1)^2|\widehat{v}_l^m|^2
=\| v\|_{H^1(\cS)}<\infty,
$$
we see from the Weierstrass M-test that 
the series $\sum_{m=-\infty}^\infty |\widehat{v}_m(\te)|^2$ 
converges locally uniformly in $\te \in (-\pi/2,\pi/2)$.
This means that the right-hand side of \eqref{v-expansion2}
is convergent in $L^2(\cE)$ locally uniformly in $\te \in (-\pi/2,\pi/2)$.
Since we know that \eqref{v-expansion2} holds in $L^2(\cS)$,
we conclude that \eqref{v-expansion2} is valid in $L^2(\cE)$ for every $\te \in (-\pi/2,\pi/2)$.
Thus \eqref{oblique-spherical-d2} is verified,
and the proof is completed.
\end{proof}

\smallskip

\section{Axially symmetric solutions of Backus problem}
\label{sec:backus-problem}

%
%

In this section, 
we finally prove the existence of axially symmetric solutions 
of Backus problem \eqref{backus-problem} near the dipole.

\begin{rem}
\label{rem:dipole}
{\rm
Notice that, since $|\na d(x)|=\sqrt{1+3\,x_3^2}$ for $x\in\cS$, then 
 $|\na d| \in H^s(\cS)$ for any $s$, thanks to Proposition \ref{prop:C-infty-in-Hs}, being as
$|\na d|\in C^\infty(\cS)$. 
}
\end{rem}

\begin{lem}
\label{lem:products}
Let $u,v$ be harmonic functions in $\RR^N\setminus\ol{B}$, continuous up to the boundary $\cS$, and such that $u, v\to 0$ as $|x|\to\infty$. If $u,v\in H^{s+1}_{\rm ax}(\cS)$ for $s>1$. 
Then, $|\nabla u|^2 ,|\nabla v|^2\in H^s_{\rm ax}(\cS)$ and 
\begin{align*}
 \||\nabla u|^2\|_{H^s(\cS)}&\leq C\|u\|_{H^{s+1}(\cS)}^2,\\
 \||\nabla u|^2-|\nabla v|^2\|_{H^s(\cS)}&\leq C\left(\|u\|_{H^{s+1}(\cS)}+\|v\|_{H^{s+1}(\cS)}\right)\|u-v\|_{H^{s+1}(\cS)},
\end{align*}
for some constant $C>0$ independent of $u,v$. 
\end{lem}

\begin{proof}
The assertion follows from the decompositions
\begin{align*}
 |\nabla u|^2&=u_r^2+u_\theta^2,\\
 |\nabla u|^2-|\nabla v|^2&=(u+v)_r(u-v)_r+(u+v)_\theta(u-v)_\theta, 
\end{align*}
Lemmas \ref{cor-r_products} and \ref{lem-theta_products}. 
\end{proof}


We are now in position to prove our main result.

\begin{proof}[Proof of Theorem \ref{thm:existence-Sobolev}]
In the proof $C$ denotes a generic positive constant depending only on $s$.
We define operators $\cT$ and $\Psi$ by
\begin{equation*}
\cT [f]= |\nabla v|^2,
\qquad
\Psi[f]=\frac{1}{2} \left( g^2 -|\nabla d|^2 -\cT [f] \right),
\end{equation*}
where $v$ is a unique solution of \eqref{oblique-spherical}.
Due to Theorem~\ref{thm:oblique-solution},
Remark~\ref{rem:dipole} and Lemma~\ref{lem:products}, 
we see that $\cT$ and $\Psi$ are defined as mappings
from $H^s_{\rm ax}(\cS)$ to $H^s_{\rm ax}(\cS)$.
Put $\de=\|g-|\nabla d|\|_{H^s(\cS)}+|h|$ 
and define a closed subset $X_\de$ of $H^s(\cS)$ by
\begin{equation*}
X_\de=\{ f \in H^s_{\rm ax}(\cS): \| f\|_{H^s(\cS)} \le M\de \}.
\end{equation*}
We shall prove that $\Psi$ has a unique fixed point, by showing that it is a contraction mapping on $X_\de$, for some number $M>0$.

To this end, we observe that Theorems \ref{prop-products} and \ref{thm:oblique-solution},
Remark~\ref{rem:dipole} and Lemma~\ref{lem:products} give that
\begin{gather*}
\| g^2 -|\nabla d|^2\|_{H^s(\cS)}
\le C\| g+|\nabla d|\|_{H^s(\cS)} \| g-|\nabla d|\|_{H^s(\cS)}
\le C(\de +1)\, \de,
\\
\| \cT [f]\|_{H^s(\cS)} \le C\| v\|_{H^{s+1}(\cS)}^2 
\le C(\| f\|_{H^s(\cS)}^2 +|h_0|^2) \le C(M^2+1)\,\de^2,
\end{gather*}
for any $f \in X_\de$.
Hence, the inequality
\begin{equation}\label{Psi-estimate1}
\| \Psi[f]\|_{H^s(\cS)} \le C_1 \left[ (M^2 +1)\de +1\right] \de
\end{equation}
holds for some other positive constant $C_1$ only depending on $s$.
\par
Next, let $f_j \in X_\de$ 
and let $v_j$ be a unique solution of \eqref{oblique-spherical} for $f=f_j$ ($j=1,2$).
We see from Theorem \ref{thm:oblique-solution} and Lemma~\ref{lem:products} that
\begin{align*}
\| \cT[f_1] -\cT[f_2]\|_{H^s(\cS)} 
&\le C\left( \| v_1\|_{H^{s+1}(\cS)} +\| v_2\|_{H^{s+1}(\cS)} \right) \| v_1-v_2\|_{H^{s+1}(\cS)}
\\
&\le C(\| f\|_{H^s(\cS)} +|h_0|) \| f_1-f_2\|_{H^s(\cS)}
\\
&\le C(M+1)\,\de\, \| f_1-f_2\|_{H^s(\cS)}.
\end{align*}
Therefore we have
\begin{equation}\label{Psi-estimate2}
\| \Psi[f_1] -\Psi[f_2]\|_{H^s(\cS)} 
=\frac{1}{2} \| \cT[f_1] -\cT[f_2]\|_{H^s(\cS)} 
\le C_2\,(M+1)\,\de\, \| f_1-f_2\|_{H^s(\cS)},
\end{equation}
for a constant $C_2>0$, which only depends on $s$.
\par
Now, in order to show that $\Psi$ is a contraction mapping on $X_\de$, we must choose the positive parameters $M$ and $\de$ such that
$$
C_1 \left[ (M^2 +1)\de +1\right]<M,
$$
so that $\Psi(X_\de)\subset X_\de$ thanks to \eqref{Psi-estimate1}, and 
$$
C_2\,(M+1)\,\de<1,
$$
from \eqref{Psi-estimate2}.
 The last two inequalities are certainly satisfied if we take $M=2 C_1$ and 
\begin{equation*}
\de < \min \left\{ \frac1{1+4\,C_1^2}, \frac{1}{C_2\,(1+2\,C_1)} \right\}.
\end{equation*}
\par
Thus,
by the Banach fixed-point theorem,
$\Psi$ has a unique fixed point $f_*$ in $X_\de$.
Therefore, we can easily see that the solution $v_*$ of \eqref{oblique-spherical} with $f=f_*$ is such that $u=d+v_*$ satisfies \eqref{backus-problem} with $u=h$ on $\cE$.
Thus, the proof is completed.
\end{proof}

\begin{rem}
{\rm
The constant $h_0$ can be chosen as the average on $\cE$ of a function $h$.
Thus, loosely speaking, Theorem \ref{thm:existence-Sobolev} can be interpreted from a geophysics  point of view as: for any field intensity of dipolar character given on the Earth's surface, there exists a unique geomagnetic potential outside the Earth, with that field intensity on its surface, and with given average potential on the equator.
}
\end{rem}

\smallskip

\appendix

\section{Technical lemmas}

In this appendix, we collect the following simple lemmas for numerical sequences. In what follows, we use the standard notations for the \textit{double factorial:}
$$
n!!=\prod_{j=0}^{[n/2]-1} (n-2j),
$$
where $[\,\cdot\,]$ is the greatest integer function.

\begin{lem}\label{factorial_inequalities}
There are constants $c>0$ and $C>0$ such that
\begin{gather*}
c\,\root 4\of{n+1} \le \frac{n!!}{\sqrt{n!}} \le C\,\root 4\of{n+1},
\\
c\,\sqrt{n+1} \le \frac{(n+1)!!}{n!!} \le C\,\sqrt{n+1},
\end{gather*}
for all $n=0, 1, \cdots$.
\end{lem}
\begin{proof}
By Stirling's formula,
we can check that
\begin{eqnarray*}
&\displaystyle\lim_{n \to \infty} \frac{(2n-1)!!}{\root 4\of{2n}\, \sqrt{(2n-1)!}}=\root 4\of{\frac{2}{\pi}},
\quad
&\displaystyle\lim_{n \to \infty} \frac{(2n)!!}{\root 4\of{2n+1}\,\sqrt{(2n)!}}=\root 4\of{\frac{\pi}{2}},
\\
&\displaystyle\lim_{n \to \infty} \frac{(2n-1)!!} {\sqrt{2n-1} (2n-2)!!}=\sqrt{\frac{2}{\pi}},
\quad
&\displaystyle\lim_{n \to \infty} \frac{(2n)!!}{\sqrt{2n}\, (2n-1)!!}=\sqrt{\frac{\pi}{2}}.
\end{eqnarray*}
The desired inequalities then ensue.
\end{proof}

\begin{lem}\label{lem:recurrence-inequality}
Let the sequences of non-negative real numbers $\{ p_k\}_{k=0, 1, \dots}$
and $\{ q_k\}_{k=0, 1, \dots}$ satisfy the recurrence relations:
\begin{equation*}
\label{pk-recurrence-relation}
p_0=0,
\quad
p_k \le \si p_{k-1} +q_{k-1} \ \mbox{ for } \ k=1,2, \dots,
\end{equation*}
for some constant $0 \le \si<1$.
Then, for any non-negative numbers $\tau_1$, $\tau_2$ and $\chi$, 
there exists a positive constant $C$
depending only on $\si$, $\tau_1$ and $\tau_2$ such that
$$
\sum_{k=0}^\infty (\chi +2k+1)^{\tau_1} (2k+1)^{\tau_2} p_k^2 
\le C\sum_{k=0}^\infty (\chi +2k+1)^{\tau_1} (2k+1)^{\tau_2} q_k^2.
$$
\end{lem}
\begin{proof}
Iterating the recurrence relations gives that
$$
p_k\le \sum_{j=0}^{k-1}\si^{k-j-1} q_j \ \mbox{ for } \ k=1, 2, \dots,
$$
so that by the Cauchy-Schwarz inequality we infer that
\begin{multline*}
p_k^2 \le \left( \sum_{j=0}^{k-1} \si^{k-j-1}q_j \right)^2
\le 
\left( \sum_{j=0}^{k-1} \si^{k-j-1} \right) \left( \sum_{j=0}^{k-1} \si^{k-j-1}q_j^2 \right)
\le \\
\frac{1}{1-\si}\,\sum_{j=0}^{k-1} \si^{k-j-1}q_j^2 \ \mbox{ for any } \ k=1, 2, \cdots.
\end{multline*}
Next, we compute that
\begin{multline*}
\sum_{k=0}^\infty (\chi +2k+1)^{\tau_1} (2k+1)^{\tau_2} p_k^2
\le \\
\frac{1}{1-\si}\, \sum_{k=1}^\infty \sum_{j=0}^{k-1} (\chi +2k+1)^{\tau_1} (2k+1)^{\tau_2}\si^{k-j-1}q_j^2= \\
\frac{1}{1-\si}\, \sum_{j=0}^\infty \sum_{i=0}^\infty (\chi +2j+2i+3)^{\tau_1} (2j+2i+3)^{\tau_2}\si^i q_j^2,
\end{multline*}
after switching the two sums in the second line and then setting $k=i+j$.
 Finally, we apply the simple algebraic inequality $a+b+3 \le (a+1)(b+3)$ (for $a,b \ge 0$), and hence infer:
\begin{multline*}
\sum_{k=0}^\infty (\chi +2k+1)^{\tau_1} (2k+1)^{\tau_2} p_k^2 \le \\
\frac{1}{1-\si}\, \Biggl[ \sum_{i=0}^\infty (2i+3)^{\tau_1+\tau_2} \si^i \Biggr]
\Biggl[ \sum_{j=0}^\infty (\chi +2j+1)^{\tau_1} (2j+1)^{\tau_2} q_j^2 \Biggr].
\end{multline*}
Thus, the lemma follows.
\end{proof}

\smallskip

We conclude this appendix by recalling a well-known result for the standard space $l^p$ ($p\ge 1$) of numerical sequences $\{a_n\}_{n=0, 1, \dots}$ such that
$$
\sum_{n=0}^\infty |a_n|^p<\infty.
$$ 

\begin{lem}
\label{lem-convolution_inequality}
Let ${\mathbf{p}}=\{p_i\}_{i=0, 1,\dots}\in l^1$, ${\bf q}=\{q_j\}_{j=0, 1, \dots}\in l^2$ and let the $3$-indices sequence $\{r_{i,j,k}\}_{i, j, k=0, 1, \dots}$ satisfy
\begin{equation*}
 M:=\max\left\{\sup_{i,k}\sum_{j=0}^\infty|r_{i,j,k}|,\ \sup_{i,j}\sum_{k=0}^\infty|r_{i,j,k}|\right\}<\infty. 
\end{equation*}
Then, the sequence ${\bf s}=\{s_k\}_{k=0}^\infty$ defined by
\begin{equation*}
 s_k=\sum_{i, j=0}^\infty r_{i,j,k}\,p_{i}q_{j}
\end{equation*}
belongs to $l^2$ and satisfies
\begin{equation*}
 \|{\bf s}\|_{l^2}\leq M\|{\bf p}\|_{l^1}\|{\bf q}\|_{l^2}. 
\end{equation*}
\end{lem}

\begin{proof}
The Cauchy-Schwarz inequality gives that
\begin{multline*}
\|{\bf s}\|_{l^2}^2=\lan {\bf s}, {\bf s}\ran_{l^2}\le \sum_{k=0}^\infty\,\sum_{i, j=0}^\infty|r_{i,j,k}\,p_{i}q_{j}s_k| \le \\
\sum_{i=0}^\infty|p_{i}|\sqrt{\sum_{j=0}^\infty\sum_{k=0}^\infty |r_{i,j,k}||q_{j}|^2}\, \sqrt{\sum_{k=0}^\infty \sum_{j=0}^\infty|r_{i,j,k}||s_{k}|^2} \le 
M\|{\bf p}\|_{l^1}\|{\bf q}\|_{l^2}\|{\bf s}\|_{l^2}. 
\end{multline*}
The claim then follows at once. 
\end{proof}

\smallskip

\section*{Acknowledgements}

The first author was partially supported by the Grant-in-Aid for Early-Career Scientists 19K14574, Japan Society for the Promotion of Science.
The second author was partially supported by the Gruppo Nazionale per l'Analisi Matematica, la Probabilit\`a e le loro Applicazioni (GNAMPA) dell'Istituto Nazionale di Alta Matematica (INdAM). 
The third author was supported in part by the Grant-in-Aid for Scientific Research (C) 20K03673, Japan Society for the Promotion of Science. 
This research started while the second author was visiting the Department of Mathematics of Tokyo Institute of Technology. He wants to thank their kind hospitality.

\end{document}